\documentclass[smallextended,numbook]{svjour3}
\usepackage{amsmath}
\usepackage{amssymb,amsbsy}
\usepackage{amsfonts}
\usepackage{mathrsfs}
\usepackage{color}

\smartqed  

\newcommand{\bvarphi}{\boldsymbol{\varphi}}
\newcommand{\bpsi}{\boldsymbol{\psi}}
\newcommand{\brho}{\boldsymbol{\rho}}
\newcommand{\bsigma}{\boldsymbol{\sigma}}
\newcommand{\bu}{\mathbf{u}}
\newcommand{\bv}{\mathbf{v}}
\newcommand{\bg}{\mathbf{g}}

\def\Re{\mathrm{Re}\,}

\def\bs{\boldsymbol }

\def\mi{\mathrm{i}}

\def\dt{{\Delta t}}
\def\half{\tfrac12}

\begin{document}
\title{Stable numerical coupling of exterior and interior problems for the wave equation}
\author{Lehel Banjai \and Christian Lubich \and Francisco-Javier Sayas}
\institute{Lehel Banjai \at School of Mathematical \& Computer Sciences, Heriot-Watt University, EH14 4AS Edinburgh, UK, \email{l.banjai@hw.ac.uk} 
\and 
Christian Lubich \at Mathematisches Institut, Universit\"at T\"ubingen, Auf der Morgenstelle 10, D-72076 T\"ubingen, Germany \email{lubich@na.uni-tuebingen.de}
\and
Francisco-Javier Sayas \at Department of Mathematical Sciences, University of Delaware, Newark DE, 19716 USA \email{fjsayas@udel.edu}}
\date{Received: date / Accepted: date}
\maketitle

\abstract{The acoustic wave equation on the whole three-dimens\-ional space is considered with initial data and inhomogeneity having support in a bounded domain, which need not be convex. We propose and study a numerical method that approximates the solution using computations only in the interior domain and on its boundary. The transmission conditions between the interior and exterior domain are imposed by a time-dependent boundary integral equation coupled to the wave equation in the interior domain. We give a full discretization by finite elements and leapfrog time-stepping in the interior, and by boundary elements and convolution quadrature on the boundary. The direct coupling becomes stable on adding a stabilization term on the boundary. The derivation of stability estimates is based on a strong positivity property of the Calderon boundary operators for the Helmholtz and wave equations and uses energy estimates both in time and frequency domain. The stability estimates together with bounds of the consistency error yield optimal-order error bounds of the full discretization.}


\section{Introduction}
Boundary conditions that yield the restriction of the solution to the whole-space equation on a bounded domain are known as transparent boundary conditions. For the three-dimensional wave equation they are nonlocal in space and time.
In the last decades, a vast literature on approximating transparent boundary conditions has developed. There are fast algorithms for implementing the exact, nonlocal boundary conditions in special domains such as balls (e.g., Grote \& Keller \cite{GroK}, Hagstrom \cite{Hag}, Alpert, Greengard \& Hagstrom \cite{AlpGH}, Lubich \& Sch\"adle \cite{LubS}), there are local absorbing boundary conditions (e.g., Engquist \& Majda \cite{EngM},  Hagstrom, Mar-Or \& Givoli \cite{HagMG}), there are methods based on the pole condition (Ruprecht, Sch\"adle, Schmidt \& Zschiedrich \cite{RupSSZ}, Gander \& Sch\"adle \cite{GanS}), and -- perhaps most widely used -- there are perfectly matched layers (Berenger \cite{Ber} and countless papers thereafter) that implement approximate transparent boundary conditions. None of the local methods works, however, on non-convex domains where waves may leave and re-enter the domain. While the computational domain can in principle be enlarged to become convex or even a ball, this may require the discretization of a substantially larger domain than the domain of physical interest (for example, in the case of a scaffolding-like structure).

It is the objective of the present work to present a stable and convergent fully discrete algorithm that couples a standard discretization in the interior domain (by finite elements with explicit leapfrog time-stepping) with a direct discretization of the boundary integral terms (by boundary elements and convolution quadrature), without any requirement of convexity of the domain. The solution in the exterior domain can then be evaluated at specific points of interest by evaluating boundary integrals, which are again discretized by (the same) boundary elements and convolution quadrature.

This paper is thus related to work on boundary integral equations for the wave equation, which have attracted considerable interest in recent years. Basic analytical theory is provided by Bamberger \& Ha Duong \cite{BamH}, Lubich \cite{Lub94}, and Laliena \& Sayas \cite{LalS}. The standard discretization in space is by boundary elements (in their Galerkin or collocation variants). Two classes of discretizations in time are known to yield guaranteed stability: the space-time Galerkin approach (Ha Duong \cite{HaD}, Ha Duong,  Ludwig \& Terrasse \cite{HaDLT}) and convolution quadrature (Lubich \cite{Lub94} and more recently
Hackbusch, Kress \& Sauter \cite{HacKS}, Banjai \& Sauter \cite{BanS}, Banjai \cite{Ban10}, Banjai, Lubich \& Melenk \cite{BanLM}, Chappell \cite{Cha}, Chen, Monk, Wang \& Weile \cite{CheMWW}, Monegato, Scuderi \& Stani\'c \cite{MonSS}). Here we use convolution quadrature for time discretization of the boundary integrals.

To our knowledge, the only work, containing analysis, that numerically couples boundary integral operators with the wave equation in the interior domain to implement transparent boundary conditions, is the recent paper by 
Abboud, Joly, Rodr{\'{\i}}\-guez \& Terrasse \cite{AbbJRT}. They use a first-order weak formulation of the wave equation in the interior (that we shall also adopt), which is discretized by discontinuous finite elements in space and the explicit midpoint rule in time. Their discretization of the boundary integral operators follows the space-time Galerkin framework. On the theoretical side, they show partial stability (excluding the effect of boundary perturbations), which is, however, not sufficient to obtain convergent error bounds. The partial stability result is based on a non-negativity property of the Calderon operator for the wave equation, which is also established in \cite{AbbJRT}. We refer here also to an early, purely numerical, work by Jiao, Li, Michielssen \& Min \cite{JiaLMJ}.

While our approach in this paper is clearly influenced by \cite{AbbJRT}, we choose different numerical methods and use different analytical tools to study them, and we obtain strong stability results that enable us to prove convergence and error bounds for the full discretization.
As a key analytical result, we show a strong positivity (or coercitivity) property of the Calderon operator, which we prove first for the Helmholtz equation (that is, the Laplace transformed wave equation) and then transfer it to the wave equation via an operator-valued version of the classical Herglotz theorem. The required extensions of this theorem are formulated in the preparatory Section~2, both in a time-discrete and time-continuous setting. We also show that convolution quadrature time discretization inherits the positivity property from the time-continuous to the time-discrete setting. In Section 3 we study the Calderon operator of the Helmholtz equation, showing the positivitiy property that we transfer to the wave equation in Section 4. There we also describe the weak first-order formulation of the coupled problem that we adopt from \cite{AbbJRT}.

In Section 5 we describe the discretization that we propose and study. Space discretization is done by standard finite elements in the interior domain and by boundary elements. Time discretization is by standard leapfrog time stepping in the interior, and by convolution quadrature on the boundary. The coupling is stabilized by adding an extra term to the naive coupling of the methods. The fully discrete method remains explicit in the interior and is implicit only in the boundary variables, for which a linear system with the same positive definite matrix is solved in each time step.

In Section 6 we study the stability of the spatial semi-discretization. The strong positivity property of the Calderon operator, which is inherited by the Galerkin boundary element space discretization, is a key aspect. We use energy estimates both in the time-dependent equations and in the Laplace-transformed (frequency-domain) equations. Combining our stability estimates with bounds of the consistency error then allows us to obtain optimal-order error bounds of the semi-discretization in Section 7.

In Sections 8 and 9 we carry out an analogous, but technically more demanding programme for the full discretization. We make essential use of the fact that the strong positivity property is preserved under convolution quadrature time discretization. Our final result, Theorem 9.1, yields an asymptotically optimal $O(h+\dt^2)$ error bound in the natural norms for linear finite elements and naturally mixed piecewise linear / piecewise constant boundary elements, under the usual CFL condition for the leapfrog method and for the convolution quadrature based on the second-order backward difference formula. The spatial order can be increased with finite elements and boundary elements of higher degree.

\section{Preparation: Variants of the Herglotz theorem}

A key ingredient of the analysis of both the continuous and discretized wave equation is the positivity of a boundary integral operator and its discretization. This positivity resides on an operator-valued variant of the classical Herglotz theorem \cite{Her}, which states that an analytic function has positive real part on the unit disc if and only if convolution with its coefficient sequence is a positive semidefinite operation.

\subsection{A time-discrete operator-valued Herglotz theorem}

Let $V$ be a complex Hilbert space with dual $V'$, with the anti-duality  denoted by $\langle \cdot, \cdot \rangle$.  Let $B(\zeta): V \rightarrow V'$ and $R(\zeta) :  V \rightarrow V$ be  analytic families of bounded linear operators for $|\zeta| \leq \rho$.  We assume the uniform bounds
\begin{equation}
  \label{eq:bound_B_zeta}
  \|B(\zeta)\|_{V' \leftarrow V} \leq M, \quad \|R(\zeta)\|_{V \leftarrow V} \leq M,  \qquad |\zeta| \leq \rho,
\end{equation}
and expand $B(\zeta)$ and $R(\zeta)$ as
\[
B(\zeta) = \sum_{n = 0}^\infty B_n \zeta^n, \qquad
R(\zeta) = \sum_{n = 0}^\infty R_n \zeta^n.
\]
\begin{lemma}\label{lemma:disc_Herglotz}
In the above situation the following statements are equivalent:
\begin{enumerate}
\item 
$
\displaystyle
\Re \langle w, B(\zeta) w\rangle \geq \gamma \| R(\zeta) w\|^2, \qquad  \forall w \in V, |\zeta| \leq \rho.
$
\item $
\displaystyle
 \sum_{n = 0}^\infty \rho^{2n} \Re \left\langle w_n, \sum_{j = 0}^n B_{n-j}w_j\right\rangle \geq \gamma \sum_{n = 0}^\infty \rho^{2n} \left\|\sum_{j = 0}^n R_{n-j}w_j\right\|^2
$
holds for any finite sequence $w_n \in V$.
\end{enumerate}
\end{lemma}
\begin{proof}
  Let $\widehat B(\theta) = B(\rho e^{\mi \theta}), \widehat R(\theta) = R(\rho e^{\mi \theta})$ and for any finite sequence $(w_n)$ let $\widehat w(\theta) = \sum_{n = 0}^\infty e^{\mi n \theta}\rho^n w_n$.  Then by Parseval's formula we have
\[
\sum_{n = 0}^\infty  \langle \rho^n w_n , \sum_{j = 0}^n \rho^{n-j} B_{n-j} \rho^j w_j\rangle = \int_{-\pi}^\pi \langle \widehat w(\theta), \widehat B(\theta) \widehat w(\theta) \rangle d\theta
\]
and 
\[
\int_{-\pi}^\pi \|\widehat R(\theta)\widehat w(\theta)\|^2 d\theta = \sum_{n = 0}^\infty \rho^{2n} \left\|\sum_{j = 0}^n R_{n-j}w_j\right\|^2,
\]
which yields the implication $1. \implies 2.$  For the reverse direction one additionally  uses a sequence of non-negative approximate $\delta$-functions $p_n(\theta)$ (e.g., the Fej\'er sequence) and chooses $\widehat w(\theta)=p_n(\theta-\theta_*)^{1/2} w_*$ to localize the above integrals near an arbitrary $\theta_*$.
\qed \end{proof}

\subsection{A time-continuous operator-valued Herglotz theorem}\label{section:cont_herglotz}

Let $B(s): V \rightarrow V'$ and $R(s) : V \rightarrow V$ be analytic families of bounded linear operators for $\Re s \geq \sigma$.  We assume the uniform bounds
\begin{equation}
  \label{eq:bound_B_s}
  \|B(s)\|_{V' \leftarrow V} \leq M |s|^\mu, \quad \|R(s)\|_{V \leftarrow V} \leq M |s|^\mu, \qquad \Re s \geq \sigma.
\end{equation}
For integer $m > \mu+1$, we define the integral kernel 
\begin{equation}
  \label{eq:B_kernel_m}
  K_m(t) = \frac1{2\pi\mi} \int_{\sigma + \mi \mathbb{R}} e^{st} s^{-m} B(s) ds.
\end{equation}
For a function $w \in C^m([0,T], V)$ with $w(0) = w'(0)= \dots = w^{m-1}(0) = 0$,  we let 
\[
\bigl(B(\partial_t)w \bigr) (t) = \left(\frac{d}{dt}\right)^m \int_0^t K_m(t-\tau) w(\tau) d\tau.
\]
We note that $B(\partial_t)w$ is the distributional convolution of the inverse Laplace transform of $B(s)$ with $w$.
\begin{lemma}\label{lemma:cont_Herglotz}
  In the above situation the following statements are equivalent:
  \begin{enumerate}
  \item $\Re \langle w, B(s) w\rangle \geq \gamma \|R(s) w\|^2, \qquad \forall w \in V, \; \Re s \geq \sigma$.
\item $ \displaystyle \int_0^\infty e^{-2\sigma t} \Re \langle w(t) , B(\partial_t) w (t) \rangle dt \geq \gamma \int_0^\infty e^{-2\sigma t} \|R(\partial_t) w(t)\|^2 dt$, for all  $w \in C^m([0,\infty), V)$ with finite support, $w(0) = w'(0)= \dots = w^{m-1}(0) = 0$, and for all $t \geq 0$.
  \end{enumerate}
\end{lemma}
\begin{proof}
  Similarly as above the result is obtained using Plancherel's formula, which here gives
\[
\int_0^\infty  \langle  e^{-\sigma t} w(t), e^{-\sigma t} B(\partial_t) w (t) \rangle dt
 = \int_{\sigma+\mi \mathbb{R}} \langle \mathscr{L} w (s), B(s)  \mathscr{L} w (s)\rangle ds,
\]
where $\mathscr{L} w$ denotes the Laplace transform of $w$, and
\[
\int_{\sigma+\mi \mathbb{R}} \|R(s) \mathscr{L} w (s)\|^2 ds
= \int_0^{\infty} e^{-2\sigma t} \|R(\partial_t) w (t)\|^2 dt.
\]
\qed \end{proof}

\subsection{Convolution quadrature and preserving the positivity}

Convolution quadrature based on an A-stable multistep method discretizes $B(\partial_t) w (t)$  by a discrete convolution 
\[
\left(B(\partial_t^{\Delta t}) w\right) (n \Delta t) = \sum_{j = 0}^n B_{n-j} w(j\Delta t). 
\]
Here the weights $B_n$ are defined as the coefficients of the power series 
\[
B\left(\frac{\delta(\zeta)}{\Delta t}\right) = \sum_{n = 0}^\infty B_n \zeta^n,
\]
where in this paper we choose  $\delta(\zeta)$ to be the generating function of the second order backward difference formula (BDF2):
\[
\delta(\zeta) = (1-\zeta) + \tfrac12(1-\zeta)^2.
\]
The method is of order 2, which can be formulated as 
\[
\delta(e^{-z}) = z + O(z^3)
\]
and it is strongly A-stable, which means that
\[
\Re \delta(\zeta) \geq  \alpha + O(\alpha^2), \qquad |\zeta| \leq e^{-\alpha},
\]
for small $\alpha$. It is known that
\begin{equation}\label{cq-conv}
B(\partial_t^{\Delta t}) w (t)-B(\partial_t) w(t) = O(\Delta t^2), \quad \text{ uniformly for } t = n\Delta t \leq T,
\end{equation}
for sufficiently smooth functions $w$ with sufficiently many vanishing derivatives at $t=0$, see \cite{Lub94} for details. Moreover the scheme preserves the positivity property of the continuous convolution.

\begin{lemma}\label{lemma:cq_Herglotz}
  In the situation of Lemma~\ref{lemma:cont_Herglotz} the condition 1.\ of that lemma implies, for $\sigma \Delta t >0$ small enough and with a $\rho = e^{-\sigma \Delta t}+O(\dt^2)$,
\[
\sum_{n = 0}^\infty \rho^{2n} \Re \langle w(n\Delta t), B(\partial_t^{\Delta t})w(n\Delta t)\rangle \geq \gamma \sum_{n = 0}^\infty \rho^{2n} \|R(\partial_t^{\Delta t})w(n\Delta t)\|^2,
\]
for any function $w: [0,\infty) \rightarrow V$ with finite support.
\end{lemma}
\begin{proof}
  Under the above conditions we have
\[
\Re \left \langle w, \left(\sum_{n = 0}^\infty B_n \zeta^n \right) w \right\rangle = 
\Re \left\langle w, B\left(\frac{\delta(\zeta)}{\Delta t}\right)w \right \rangle
\geq \gamma \left\|R\left(\frac{\delta(\zeta)}{\Delta t}\right)w\right\|^2, 
\]
for all $w \in V$ and $|\zeta| \leq \rho$. The result then follows from Lemma~\ref{lemma:disc_Herglotz}.
\qed \end{proof}

\section{Calderon operator for the Helmholtz equation}\label{section:cald_Helm}

With the Helmholtz equation
\begin{equation}
  \label{eq:helmholtz}  
  s^2 u -\Delta u = 0, \qquad  x \in \mathbb{R}^3\setminus \Gamma,
\end{equation}
and the boundary surface $\Gamma$ of a bounded Lipschitz domain $\Omega \subset \mathbb{R}^3$ we associate the usual boundary integral potentials \cite{LalS}: the single layer potential 
\[
S(s) \varphi (x) = \int_\Gamma \frac{e^{-s|x-y|}}{4\pi |x-y|} \varphi(y) d\Gamma_y,\qquad x \in \mathbb{R}^3 \setminus \Gamma,
\]
the double layer potential
\[
D(s) \varphi (x) = \int_\Gamma \left(\partial_{n_y}\frac{e^{-s|x-y|}}{4\pi |x-y|}\right) \varphi(y) d\Gamma_y, \qquad x \in \mathbb{R}^3 \setminus \Gamma, 
\]
where $\partial_{n_y}$ denotes the exterior normal derivative with respect to the variable $y$. The corresponding boundary integral operators are defined as
\begin{align}
V(s) \varphi (x) &= \int_\Gamma \frac{e^{-s|x-y|}}{4\pi |x-y|} \varphi(y) d\Gamma_y, &  x &\in \Gamma,\\
K(s) \varphi (x) &= \int_\Gamma \left(\partial_{n_y}\frac{e^{-s|x-y|}}{4\pi |x-y|}\right) \varphi(y) d\Gamma_y, & x &\in \Gamma,\\
K^T(s) \varphi (x) &= \partial_{n_x}\int_\Gamma \frac{e^{-s|x-y|}}{4\pi |x-y|} \varphi(y) d\Gamma_y, & x &\in \Gamma,\\
W(s) \varphi (x) &= -\partial_{n_x}\int_\Gamma \left(\partial_{n_y}\frac{e^{-s|x-y|}}{4\pi |x-y|}\right) \varphi(y) d\Gamma_y, & x &\in \Gamma.
\end{align}
The above boundary integral operators are bounded linear operators on the following spaces
\begin{align*}
  V(s) &: H^{-1/2}(\Gamma) \rightarrow H^{1/2}(\Gamma), & K(s) &: H^{1/2}(\Gamma) \rightarrow H^{1/2}(\Gamma),\\
  K^T(s) &: H^{-1/2}(\Gamma) \rightarrow H^{-1/2}(\Gamma), & W(s) &: H^{1/2}(\Gamma) \rightarrow H^{-1/2}(\Gamma).
\end{align*}
with the following bounds holding for all $\Re s \geq \sigma > 0$
\begin{align*}
  \|V(s)\|_{H^{1/2}(\Gamma) \leftarrow H^{-1/2}(\Gamma)} &\leq C(\sigma) |s|, 
  \\
  \|K(s)\|_{H^{1/2}(\Gamma) \leftarrow H^{1/2}(\Gamma)} &\leq C(\sigma) |s|^{3/2},\\
  \|K^T(s)\|_{H^{-1/2}(\Gamma) \leftarrow H^{-1/2}(\Gamma)} &\leq C(\sigma) |s|^{3/2}, 
  \\
  \|W(s)\|_{H^{-1/2}(\Gamma) \leftarrow H^{1/2}(\Gamma)} &\leq C(\sigma) |s|^2.
\end{align*}
For a proof of these facts see \cite{BamH,BamH2} and for a table with all these properties listed see \cite{LalS}. We note that $C(\sigma)$ depends polynomially on~$\sigma^{-1}$.

Let $\gamma^-$ and $\gamma^+$ denote the interior and exterior traces on the boundary $\Gamma$, whereas $\partial_n^-$  and $\partial_n^+$  the interior and exterior normal traces on  $\Gamma$. Further we will also denote by $\Omega^+ = \mathbb{R}^3 \setminus \overline{\Omega}$ the domain exterior to $\Omega$.  The relationship between the boundary integral potentials and operators is given by
\[
V(s) \varphi= \gamma^- S(s) \varphi = \gamma^+ S(s) \varphi, \qquad
K(s) \varphi = \{\{D(s)\varphi\}\}
\]
and
\[
K^T(s) \varphi= \{\{\partial_n S(s) \varphi\}\}, \qquad
W(s) \varphi = - \partial^-_n D(s)\varphi = -\partial^+_n D(s)\varphi,
\]
where $\{\{ \gamma u\}\} = \tfrac12 (\gamma^- u+\gamma^+ u)$ denotes the average of the jump accross the boundary.  

In terms of these operators the solution of the Helmholtz equations is expressed as
\[
u = sS(s)\varphi + D(s) \psi,
\]
where
\[
\varphi = [\tfrac1s \partial_n u] , \qquad \psi = -[\gamma u],
\]
and  $[\gamma u] = \gamma^- u - \gamma^+ u$, $[\partial_n u] = \partial^-_n u- \partial^+_n u$ denote the jumps in the boundary traces.
Next we define a Calderon operator, whose positivity will be crucial for the analysis:
\begin{equation}
  \label{eq:calderon}
B(s)   =
\begin{pmatrix}
  sV(s) & K(s) \\ -K^T(s) & \tfrac1s W(s)
\end{pmatrix}.
\end{equation}
 In the following we  denote the anti-duality  between $H^{-1/2}(\Gamma) \times H^{1/2}(\Gamma)$ and $H^{1/2}(\Gamma) \times H^{-1/2}(\Gamma)$ by $\langle \cdot, \cdot \rangle_\Gamma$.

\begin{lemma}\label{lemma:calderon_pos}
There exists $\beta > 0$ so that the Calderon operator \eqref{eq:calderon} satisfies
\[
\Re \left \langle
  \begin{pmatrix}
      \varphi \\ \psi
  \end{pmatrix}, B(s) 
  \begin{pmatrix}
      \varphi \\ \psi
  \end{pmatrix}
\right \rangle_\Gamma
 \geq \beta\, \min(1,|s|^2) \frac{\Re s}{|s|^2} \left(\|\varphi\|^2_{H^{-1/2}(\Gamma)} + \|\psi\|^2_{H^{1/2}(\Gamma)}\right)
\]
for  $\Re s > 0$ and for all $\varphi \in H^{-1/2}(\Gamma)$ and $\psi \in H^{1/2}(\Gamma)$.
\end{lemma}
\begin{proof}
From the identities 
\[
\{\{\gamma u\}\} = sV(s) \varphi + K(s) \psi, \qquad 
\{\{\partial_n u\}\} = sK^T(s) \varphi - W(s) \psi, \qquad 
\]
it follows that
\begin{equation}
  \label{eq:Bs_eq}  
B(s)
\begin{pmatrix}
  \varphi \\ \psi
\end{pmatrix}
=
\begin{pmatrix}
  \{\{\gamma u\}\}\\
- \tfrac1s\{\{\partial_n u\}\}
\end{pmatrix}
.
\end{equation}
Hence, using Green's theorem and \eqref{eq:helmholtz},
\begin{align*}  
\Re \left \langle
  \begin{pmatrix}
      \varphi \\ \psi
  \end{pmatrix}, B(s) 
  \begin{pmatrix}
      \varphi \\ \psi
  \end{pmatrix}
\right \rangle_\Gamma
&= \Re \langle [\tfrac1s \partial_n u], \{\{\gamma u\}\}\rangle_\Gamma
+ \Re \langle \{\{\tfrac1s \partial_n u\}\}, [\gamma u]\rangle_\Gamma\\
&=\Re \langle \tfrac1s \partial_n^- u, \gamma^-u \rangle_\Gamma+\Re \langle -\tfrac1s \partial_n^+ u, \gamma^+u \rangle_\Gamma\\
&= \Re s\left( \|\tfrac1s \nabla u\|_{L_2(\mathbb{R}^3 \setminus \Gamma)}^2 + \|u\|_{L_2(\mathbb{R}^3 \setminus \Gamma)}^2\right)  \\
& \geq \beta\, \min(1,|s|^2) \frac{\Re s}{|s|^2} \left(\|\varphi\|^2_{H^{-1/2}(\Gamma)} + \|\psi\|^2_{H^{1/2}(\Gamma)}\right).
\end{align*}
The final inequality above is obtained using the trace inequalities as follows:
\begin{align*}
  \|\varphi\|^2_{H^{-1/2}(\Gamma)} & = \left\|[\tfrac1s \partial_n u]\right\|^2_{H^{-1/2}(\Gamma)} \leq 
C \left(\|\tfrac1s \nabla u\|^2_{L_2(\mathbb{R}^3\setminus \Gamma)^3}+\|\tfrac1s \Delta u\|^2_{L_2(\mathbb{R}^3\setminus \Gamma)}\right)\\
& = C \left(\|\tfrac1s \nabla  u\|^2_{L_2(\mathbb{R}^3\setminus \Gamma)^3}+\|s u\|^2_{L_2(\mathbb{R}^3\setminus \Gamma)}\right)\\
&\leq C |s|^2 \max(1,|s|^{-2}) \left(\|\tfrac1s \nabla  u\|^2_{L_2(\mathbb{R}^3\setminus \Gamma)^3}+\|  u\|^2_{L_2(\mathbb{R}^3\setminus \Gamma)}\right)
\end{align*}
and similarly for $\psi = -[\gamma u]$.
\qed \end{proof}

\section{Boundary integral formulation of the wave equation}

\subsection{Calderon operator for the wave equation}
Consider the wave equation in $\mathbb{R}^3$
\begin{equation}
  \label{eq:wave_allspace}  
\begin{aligned}
\partial_t^2 u -\Delta u &= \dot f & &\text{in } \mathbb{R}^3 \times [0,T],\\
u(x,0 ) = u_0, \quad \partial_tu(x,0 ) &= v_0, & &\text{in } \mathbb{R}^3.
\end{aligned}
\end{equation}
Let again $\Omega \subset \mathbb{R}^3$ be a bounded Lipschitz domain with boundary $\Gamma$ and assume that the supports of $u_0,v_0$, and $\dot f$ are contained in $\Omega$.

We can rewrite (\ref{eq:wave_allspace}) as a problem set on the interior domain
\begin{equation}
  \label{eq:wave_int}
\begin{aligned}
 \partial_t^2 u^- -\Delta u^- &= \dot f & &\text{in } \Omega \times [0,T],\\
u^-(x,0 ) = u_0, \quad \partial_tu^-(x,0 ) &= v_0, & &\text{in } \Omega,
\end{aligned}
\end{equation}
a problem set in the exterior
\begin{equation}
  \label{eq:wave_ext}
\begin{aligned}
\partial_t^2 u^+ -\Delta u^+ &= 0 & &\text{in } \Omega^+ \times [0,T],\\
u^+(x,0 ) = 0, \quad \partial_tu^+(x,0 ) &= 0, & &\text{in } \Omega^+,
\end{aligned}
\end{equation}
where $\Omega^+ = \mathbb{R}^3 \setminus \overline{\Omega}$, and transmission conditions coupling the two sets of equations
\begin{equation}
  \label{eq:trans_condn}
  \gamma^-u^-=\gamma^+u^+, \qquad \partial^-_n u^-=\partial^+_n u^+.
\end{equation}
The solution of (\ref{eq:wave_allspace}) is then given by $u = u^-$ in $\Omega$ and by $u = u^+$ in $\Omega^+$. 

With time convolution operators based on the boundary integral operators for the Helmholtz equation, the solution of the exterior equations can then be written as
\begin{equation}
  \label{eq:u_bie}
  u^+ = S(\partial_t) \partial_t\varphi+ D(\partial_t) \psi.
\end{equation}
The boundary densities are given by
\[
\varphi = -\partial_t^{-1} \partial^+_n u^+, \qquad \psi = \gamma^+ u^+
\]
and satisfy the equation
\[
B(\partial_t)
\begin{pmatrix}
  \varphi \\ \psi 
\end{pmatrix}
= \frac12
\begin{pmatrix}
  \gamma^- u^- \\ -\partial_t^{-1} \partial_n^- u^-
\end{pmatrix},
\]
where $B(s)$ is defined in \eqref{eq:calderon}, the notation $B(\partial_t)$ is explained in Section~\ref{section:cont_herglotz}, and we have used \eqref{eq:Bs_eq} and the fact that $\gamma^- u^+ = \partial_n^- u^+ = 0$.

\subsection{Positivity of the time-dependent Calderon operator}

Applying Lemma~\ref{lemma:cont_Herglotz} we have the positivity of the time-dependent Calderon operator $B(\partial_t)$:

\begin{lemma}\label{lemma:calderon_time_pos}
  With the constant $\beta > 0$ from Lemma~\ref{lemma:calderon_pos} we have that
\[
\begin{split}  
&\int_0^T e^{-2t/T} \left \langle
 \begin{pmatrix}
      \varphi(\cdot,t) \\ \psi(\cdot,t)
  \end{pmatrix}, B(\partial_t) 
  \begin{pmatrix}
      \varphi \\ \psi
  \end{pmatrix}(\cdot,t)
\right \rangle_\Gamma dt 
\\
&\geq 
\beta\,  c_T \int_0^T e^{-2t/T}\left(\|\partial_t^{-1}\varphi(\cdot,t)\|^2_{H^{-1/2}(\Gamma)} + \|\partial_t^{-1}\psi(\cdot,t)\|^2_{H^{1/2}(\Gamma)}\right) dt,
\end{split}
\]
for any  $T > 0$ and for
 all $\varphi \in C^4([0,T],H^{-1/2}(\Gamma))$ and all  $\psi \in C^3([0,T],H^{1/2}(\Gamma))$ with $\varphi(\cdot,0) = \partial_t \varphi(\cdot,0) = \cdots = \partial_t^3 \varphi(\cdot,0) =0$, $\psi(\cdot,0) = \partial_t \psi(\cdot,0) = \partial_t^2 \psi(\cdot,0)= 0$. Here, $c_T=\min(T^{-1},T^{-3})$.
\end{lemma}
\begin{proof}
  The proof follows directly from Lemma~\ref{lemma:cont_Herglotz} and Lemma~\ref{lemma:calderon_pos}, where we use $\Re s\geq \sigma = 1/T$ and the lower bound $\min(1,|s|^2) \Re s \geq \min( T^{-1},T^{-3})$.  The smoothness requirements on $\varphi$ and $\psi$ result from the bounds on the boundary integral operators. The reason that the integrals extend only up to $T$, lies in the causality property that $B(\partial_t) w(T)$ depends only on  $w(t)$ for $ t \leq T$.
\qed \end{proof}

Similarly, Lemma~\ref{lemma:cq_Herglotz} implies the positivity of the convolution quadrature approximation $B(\partial^{\Delta t}_t)$.  This result will be needed later in the paper.

\begin{lemma}\label{lemma:4.2}
Let $E:[0,\infty)\to [0,\infty)$, $S:\mathbb R\to \mathbb R$, $\varphi \in C^4([0,T],H^{-1/2}(\Gamma))$, $\psi \in C^3([0,T],H^{1/2}(\Gamma))$, with $\varphi(\cdot,0) = \partial_t \varphi(\cdot,0) = \cdots = \partial_t^3 \varphi(\cdot,0) =0$, $\psi(\cdot,0) = \partial_t \psi(\cdot,0) = \partial_t^2 \psi(\cdot,0)= 0$. If
\begin{equation}\label{eq:lemma42}
\dot E+\left\langle\begin{pmatrix} \varphi \\ \psi\end{pmatrix}, B(\partial_t) \begin{pmatrix} \varphi \\ \psi\end{pmatrix}\right\rangle_\Gamma = S \qquad \mbox{in $[0,T]$},
\end{equation}
then
\begin{eqnarray*}
& & E(T)+\beta c_T \int_0^T \left( \| \partial_t^{-1} \varphi(\cdot,t)\|_{H^{1/2}(\Gamma)}^2+\| \partial_t^{-1} \psi(\cdot,t)\|_{H^{-1/2}(\Gamma)}^2\right) dt\\
& & \hspace{5cm} \le e^2 E(0)+\int_0^T e^{2(1-t/T)} S(t) dt,
\end{eqnarray*}
where $c_T=\min\{T^{-1},T^{-3}\}$.
\end{lemma}

\begin{proof}
The result follows from multiplying \eqref{eq:lemma42} by $e^{-2t/T}$, using that $E$ is non-negative and applying Lemma \ref{lemma:calderon_time_pos}.
\qed\end{proof}

\subsection{First-order formulation and energy estimate}

We rewrite the wave equation as a first-order system (and omit the superscript $^-$ in the interior)
\begin{equation}
  \label{eq:wave}  
\begin{array}{rcl}
\dot u &=& \nabla \cdot v + f 
\\
\dot v &=& \nabla u \qquad\qquad\quad
\end{array}\text{in }\Omega,
\end{equation}
with the coupling condition $\psi = \gamma u$, $\varphi = -\gamma v \cdot n$ 
expressed as
\begin{equation*}
B(\partial_t)
  \begin{pmatrix}
   \varphi\\ \psi
  \end{pmatrix}
= \tfrac12
\begin{pmatrix}
  \gamma u \\ -\gamma v \cdot n
\end{pmatrix}
\qquad \hbox{on } \Gamma.
\end{equation*}
As in \cite{AbbJRT}, we determine the weak formulation using
\[
(\nabla\cdot v,w) = -\half (v,\nabla w) + \half (\nabla\cdot v,w) +\half\langle n\cdot \gamma v,w\rangle_\Gamma
\]
and similarly for $(\nabla u, z)$.  Here $(\cdot, \cdot)$ denotes the inner product in $L_2(\Omega)$ or $L_2(\Omega)^3$ as appropriate.  With this weak formulation the coupled system reads
\begin{align}
&  (\dot u, w) = -\half(v,\nabla w) + \half (\nabla\cdot v,w) -\half\langle\varphi,\gamma w\rangle_\Gamma +(f,w)\\
&  (\dot v, z) = -\half(u,\nabla\cdot z) + \half (\nabla u,z) +\half\langle\psi,\gamma z\cdot n\rangle_\Gamma\\
& \left\langle
  \begin{pmatrix}
    \xi\\\eta
  \end{pmatrix}, {B} (\partial_t) 
  \begin{pmatrix}
    \varphi\\\psi
  \end{pmatrix}
\right\rangle_\Gamma = \half\langle\xi,\gamma u\rangle_\Gamma - \half \langle \gamma v\cdot n, \eta\rangle_\Gamma 
\end{align}
for all $w,z\in H^1(\Omega)$ and $(\xi,\eta)\in H^{-1/2}(\Gamma)\times H^{1/2}(\Gamma) $.
Testing with $w = u, z = v, \xi = \varphi, \eta = \psi$ and adding the three equations up we get
\[
\frac{d}{dt} \left(\half\|u\|_{L_2(\Omega)}^2+\half\|v\|_{L_2(\Omega)}^2\right)
+\left\langle
  \begin{pmatrix}
    \varphi\\\psi
  \end{pmatrix}, {B} (\partial_t) 
  \begin{pmatrix}
    \varphi\\\psi
  \end{pmatrix}
\right\rangle_\Gamma
= (f,u).\]
From the positivity property of the Calderon operator in Lemma~\ref{lemma:calderon_time_pos} it follows  that the {\it field energy} (so called because its Maxwell analogue is the electro-magnetic energy in the field)
$$
E = \half\|u\|_{L_2(\Omega)}^2+\half\|v\|_{L_2(\Omega)}^2
$$
satisfies for $t > 0$  (if $f=0$; see Lemma \ref{lemma:4.2}) 
$$
E(t)+ \beta c_t \int_0^t
\Bigl( \| \partial_t^{-1}\varphi (\cdot,\tau)\|_{H^{-1/2}(\Gamma)}^2 +\| \partial_t^{-1} \psi(\cdot,\tau) \|_{H^{1/2}(\Gamma)}^2\Bigr)d\tau \le e^2\, E(0).
$$

\section{Discretization}
\subsection{FEM--BEM spatial semidiscretization}

Let $U_h,V_h,\Psi_h, \Phi_h$ be finite dimensional subspaces of the following Sobolev spaces
\[
U_h \subset H^1(\Omega), \quad V_h = U_h^3 \subset H^1(\Omega)^3, \quad \Psi_h \subset H^{1/2}(\Gamma), \quad \Phi_h \subset H^{-1/2}(\Gamma).
\]
In particular we can choose $U_h$ as the finite element space of piecewise linear functions, $\Psi_h$ the boundary element space of piecewise linear functions, and $\Phi_h$ the boundary element space of piecewise constant functions.  The chosen bases of these spaces are denoted by $(b_i^U), (b_j^V),(b_k^\Psi)$, and $(b_\ell^\Phi)$,  respectively. We assume that $\Psi_h$ and $\Phi_h$ contain the traces of $U_h$: $
 \gamma U_h \subseteq \Psi_h,\;   \partial_n U_h \subseteq \Phi_h.
$

The semi-discretized system then reads: find $u_h(\cdot,t) \in U_h$, $v_h(\cdot,t) \in V_h$, $\varphi_h(\cdot,t)  \in \Phi_h$, $\psi_h(\cdot,t)  \in \Psi_h$ such that
\begin{align}\label{eq:semi_disc}
&  (\dot u_h, w_h) = -\half(v_h,\nabla w_h) + \half (\nabla\cdot v_h,w_h) -\half\langle\varphi_h,\gamma w_h\rangle_\Gamma +(f,w_h)\\
&  (\dot v_h, z_h) = -\half(u_h,\nabla\cdot z_h) + \half (\nabla u_h,z_h) +\half\langle\psi_h,\gamma z_h\cdot n\rangle_\Gamma\\
& \left\langle
  \begin{pmatrix}
    \xi_h\\\eta_h
  \end{pmatrix}, {B} (\partial_t) 
  \begin{pmatrix}
    \varphi_h\\\psi_h
  \end{pmatrix}
\right\rangle_\Gamma = \half\langle\xi_h,\gamma u_h\rangle_\Gamma - \half \langle \gamma v_h\cdot n, \eta_h\rangle_\Gamma 
\end{align}
for all $w_h \in U_h$, $z_h \in V_h$, $\xi_h \in \Phi_h$, and  $\eta_h\in \Psi_h$. 

For the vectors of nodal values this leads to a coupled system of ordinary differential and integral equations
\begin{eqnarray*}
\mathbf{M}_0 \dot \bu &=& -\mathbf{D}^T \bv - \mathbf{C}_0\bvarphi+\mathbf{M}_0\mathbf{f}
\\
\mathbf{M}_1 \dot \bv &=& \phantom{-}\,\mathbf{D}\bu - \mathbf{C}_1\bpsi
\\
 \mathbf{B}(\partial_t)  \begin{pmatrix}
   \bvarphi\\ \bpsi
  \end{pmatrix}
&=& 
\begin{pmatrix}
  \mathbf{C}_0^T \bu \\ \mathbf{C}_1^T \bv
\end{pmatrix}.
\end{eqnarray*}
The matrices $\mathbf{M}_0$ and $\mathbf{M}_1$ denote the symmetric positive definite mass matrices whose entries are the inner products of the basis functions of $U_h$ and $V_h$, respectively. The matrices $\mathbf{D}, \mathbf{C}_0, \mathbf{C}_1$ have the entries
\[
\mathbf{D}|_{ji} =  -\half (b_j^{V}, \nabla b_i^U) + \half(\nabla \cdot b_j^V, b_i^U),
\]
and
\[
\mathbf{C}_0|_{k i} = -\half \langle b_k^\Phi, \gamma b_i^U \rangle_\Gamma, \qquad
\mathbf{C}_1|_{\ell j} = \half \langle b_\ell^\Psi, \gamma b_j^V \cdot n\rangle_\Gamma.
\]
The matrix $\mathbf{B}(s)$ is given as
\[
\mathbf{B}(s) =
\begin{pmatrix}
  s \mathbf{V}(s) & \mathbf{K}(s)\\
-\mathbf{K}^T(s) & \tfrac1s \mathbf{W}(s)
\end{pmatrix},
\]
where the blocks are given by 
\[
\mathbf{V}(s)|_{k k'} = \langle b_k^{\Phi},  V(s) b_{k'}^{\Phi}\rangle_\Gamma, \quad
\mathbf{K}(s)|_{k\ell } = \langle b_k^{\Phi},  K(s) b_{\ell}^{\Psi}\rangle_\Gamma, \quad
\mathbf{W}(s)|_{\ell \ell'} = \langle b_\ell^{\Psi},  W(s) b_{\ell'}^{\Psi}\rangle_\Gamma.
\]

We note that differentiating the first and last equations and eliminating $\bv$ yields the second-order formulation
\begin{eqnarray*}
\mathbf{M}_0 \ddot\bu &=& -\mathbf{D}^T \mathbf{M}_1^{-1}(\mathbf{D}\bu-\mathbf{C}_1\bpsi) - \mathbf{C}_0\dot \bvarphi+\mathbf{M}_0 \dot{\mathbf{f}}
\\
 \mathbf{B}(\partial_t)  \begin{pmatrix}
  \dot \bvarphi\\ \dot \bpsi
  \end{pmatrix}
&=& 
\begin{pmatrix}
  \mathbf{C}_0^T \dot\bu \\ \mathbf{C}_1^T \mathbf{M}_1^{-1} (\mathbf{D}u-\mathbf{C}_1\bpsi)
\end{pmatrix}.
\end{eqnarray*}

\subsection{Leapfrog--convolution quadrature time discretization}

We couple the leapfrog or St\"ormer--Verlet scheme
\begin{eqnarray*}
\mathbf{M}_1 \bv^{n+1/2} &=& \mathbf{M}_1 \bv^n +\half\dt \,\mathbf{D}\bu^n - \half \dt \,\mathbf{C}_1\bpsi^n
\\
\mathbf{M}_0 \bu^{n+1} &=& \mathbf{M}_0 \bu^n - \dt \,\mathbf{D}^T \bv^{n+1/2} - \dt\, \mathbf{C}_0\bvarphi^{n+1/2}+\dt  \mathbf{M}_0\mathbf{f}^{n+1/2}
\\
\mathbf{M}_1 \bv^{n+1} &=& \mathbf{M}_1 \bv^{n+1/2} +\half\dt \,\mathbf{D}\bu^{n+1} - \half \dt \,\mathbf{C}_1\bpsi^{n+1}
\end{eqnarray*}
to convolution quadrature
\begin{eqnarray*}
\biggl[  \mathbf{B}(\partial_t^\dt)  \begin{pmatrix}
   \bvarphi\\ \bar\bpsi
  \end{pmatrix}\biggr]^{n+1/2}
&=& 
\begin{pmatrix}
  \mathbf{C}_0^T \bar \bu^{n+1/2} \\ \mathbf{C}_1^T (\bv^{n+1/2}-\alpha \dt^2\mathbf{M}_1^{-1}\mathbf{C}_1\dot\bpsi^{n+1/2})
\end{pmatrix},
\end{eqnarray*}
where $ \bar \bu^{n+1/2}= \half (\bu^{n+1}+\bu^n)$ and $ \bar \bpsi^{n+1/2}= \half (\bpsi^{n+1}+\bpsi^n)$, and where $\alpha>0$ is a stabilization parameter and $\dot\bpsi^{n+1/2}=(\bpsi^{n+1}-\bpsi^n)/\dt$. The role of the stabilization term will become clear in the stability analysis. Under the CFL condition $\dt\|\mathbf{D}\|\le 1$ we can take $\alpha=1$ to obtain a stable scheme.

\subsection{Computing the discrete solution}
Let us assume that at time-step $n$, $\bv^n,\bu^n$, $\bvarphi^{j-1/2}$, and $\bpsi^{j}$, $j = 0,\dots,n$, are known.  Using the first equation above we can compute $\bv^{n+1/2}$. In the final equation we rewrite $\bar \bu^{n+1/2}$ and $\dot \bpsi^{n+1/2}$ as
\[
\bar \bu^{n+1/2} = \half (\bu^n -\Delta t \mathbf{M}_0^{-1}\mathbf{D}^T\bv^{n+1/2}-\Delta t \mathbf{M}_0^{-1} \mathbf{C}_0 \bvarphi^{n+1/2}+ \dt \mathbf{f}^{n+1/2}+\bu^n)
\]
and
\[
\Delta t \dot \bpsi^{n+1/2} = 2\bar\bpsi^{n+1/2}-2\bpsi^n.
\]
Grouping the known and unknown quantities together we obtain an equation for $\bvarphi^{n+1/2}$ and $\bar\bpsi^{n+1/2}$:
\[
\left(\mathbf{B}_0 + \Delta t \mathbf{H} \right)
\begin{pmatrix}
  \bvarphi^{n+1/2} \\ \bar\bpsi^{n+1/2}
\end{pmatrix}
= \bs{\chi}^n,
\]
where $\mathbf{B}_0 = \mathbf{B}(\delta(0)/\dt)$, $\bs{\chi}^n$ contains known quantities  and
\[
\mathbf{H} =
\begin{pmatrix}
  \half \mathbf{C}_0^T\mathbf{M}_0^{-1}\mathbf{C}_0 &\\& 2\alpha \mathbf{C}_1^T\mathbf{M}_1^{-1}\mathbf{C}_1
\end{pmatrix}.
\]
Both $\mathbf{B}_0$  and $\mathbf{H}$ are positive definite, hence a unique solution exists. The remaining unknowns at time-step $n+1$ can then be directly obtained from the second and third equations.

\section{Stability of the spatial semidiscretization}
\label{sec:stab-semidiscrete}

\subsection{Setting of the stability analysis}
\label{subsec:stab-semidiscrete-setting}

In the following {\it analysis} we  assume that the bases of $U_h$, $V_h$, $\Psi_h$, and $\Phi_h$ are orthonormal in
$L_2(\Omega)$, $L_2(\Omega)^3$, $H^{1/2}(\Gamma)$, $H^{-1/2}(\Gamma)$, respectively, so that the corresponding inner products are just the Euclidean inner products of the coefficient vectors, which are denoted by $(\cdot,\cdot)$ for the interior variables, and by $\langle \cdot,\cdot \rangle_\Gamma$ for the boundary variables. The Euclidean norms will be denoted by $|\cdot|$. The time discretization scheme then takes the above form with the simplification that the mass matrices $\mathbf{M}_0$ and $\mathbf{M}_1$ are identity matrices. 

In this section we are interested in the propagation of spatial discretization errors. For the errors we have similar equations but with additional inhomogeneities on the right-hand side, which are the residuals on inserting a projection of the exact solution onto the finite element space into the scheme.  We then end up with the task of bounding the Euclidean norms of the solutions to the equations
\begin{equation}\label{eq:semi_disc_perturb}
  \begin{aligned}    
\dot \bu &= -\mathbf{D}^T \bv - \mathbf{C}_0\bvarphi+\mathbf{\mathbf{f}},
\\
 \dot \bv &= \phantom{-}\,\mathbf{D}\bu - \mathbf{C}_1\bpsi+\bg,
\\
 \mathbf{B}(\partial_t)  \begin{pmatrix}
   \bvarphi\\ \bpsi
  \end{pmatrix}
&= 
\begin{pmatrix}
  \mathbf{C}_0^T \bu \\ \mathbf{C}_1^T \bv
\end{pmatrix}
+
\begin{pmatrix}
  \bs{\rho}\\ \bs{\sigma}
\end{pmatrix}
  \end{aligned}
\end{equation}
in terms of the Euclidean norms of the perturbations $\mathbf{\mathbf{f}},\bg,  \bs{\rho}, \bs{\sigma}$.  

We have the positivity relation, with  $\beta>0$ independent of the gridsize,
\[
\begin{split}  
\int_0^T e^{-2t/T}& \left \langle
 \begin{pmatrix}
      \bvarphi(t) \\ \bpsi(t)
  \end{pmatrix}, \mathbf{B}(\partial_t) 
  \begin{pmatrix}
      \bvarphi \\ \bpsi
  \end{pmatrix}(t)
\right \rangle_\Gamma dt 
\\
&\geq 
\beta\, c_T \int_0^T e^{-2t/T}\left(|\partial_t^{-1}\bvarphi(t)|^2+ |\partial_t^{-1}\bpsi(t)|^2 \right)dt,
\end{split}
\]
which is inherited from the corresponding property of the Calderon operator ${B}(\partial_t)$.

\subsection{Field energy}
\label{subsec:stab-semiE}

\begin{lemma}\label{lemma:semiE}
The semi-discrete field energy
\[
E(t) = \half |\bu(t)|^2+ \half |\bv(t)|^2
\]
is bounded along the solutions of \eqref{eq:semi_disc_perturb} by
\[
\begin{split}  
E(t) \leq C(\beta) \Bigl(E(0)\Bigr. &+ t\int_0^t \left(|\mathbf{f}(\tau)|^2  + |\bg(\tau)|^2 \right)d\tau\\ &+ \Bigl.\max(t^2,t^6) \int_0^t \left( |\ddot \brho(\tau)|^2 + |\ddot \bsigma(\tau)|^2 \right) d\tau \Bigr),
\end{split}
\]
for $t > 0$. This estimate holds provided $\brho(0) = \dot\brho(0) = 0$ and $\bsigma(0)  = \dot \bsigma(0) = 0$.
\end{lemma}
\begin{proof}
(a) Taking the inner product of the first equation  in \eqref{eq:semi_disc_perturb} with $\bu$, the second with $\bv$, and the third with $(\bvarphi, \bpsi)^T$ and summing the equations we get
\[
\dot E
+\left\langle
  \begin{pmatrix}
    \bvarphi\\\bpsi
  \end{pmatrix}, \mathbf{B} (\partial_t) 
  \begin{pmatrix}
    \bvarphi\\\bpsi
  \end{pmatrix}
\right\rangle_\Gamma
= (\bu,\mathbf{\mathbf{f}})+(\bv, \bg)+\left\langle
\begin{pmatrix}
\bvarphi\\  \bpsi
\end{pmatrix},
  \begin{pmatrix}
\bs{\brho}\\ \bs{\bsigma}   
  \end{pmatrix}
\right\rangle_\Gamma.  
\]
Integrating and using the positivity property of $\mathbf{B}(\partial_t)$ (Lemma \ref{lemma:4.2}) gives that the semi-discrete field energy satisfies, for $t > 0$,
\begin{equation}
  \label{eq:energy_ineq}  
\begin{split}  
E(t)+ & \beta \min(t^{-1},t^{-3})\int_0^t
\Bigl( | \partial_t^{-1}\bvarphi (\tau)|^2 +| \partial_t^{-1} \bpsi(\tau) |^2\Bigr)d\tau \\
\le & e^2 \Big( E(0) + \int_0^t |(\bu(\tau), \mathbf f(\tau))+(\bv(\tau), \bg(\tau))|d\tau\\
& +\int_0^t |\langle \bvarphi(\tau),\bs{\brho}(\tau)\rangle_\Gamma+\langle\bpsi(\tau),\bs{\bsigma}(\tau) \rangle_\Gamma|d\tau
\Big).
\end{split}
\end{equation}
This estimate is however not sufficient in order to estimate $E(t)$ in terms of $E(0)$ and the perturbations $\mathbf{f}$, $\bg$, $\brho$ and $\bsigma$.  While $\mathbf{f}$ and $\bg$ pose no problems, the dependence on the boundary perturbations $\brho$ and $\bsigma$ needs to be treated in a different way. 

(b) If we first assume that $\brho$ and $\bsigma$ are zero, then using the Cauchy-Schwarz inequality 
\[
e^2|(\mathbf{f}(\tau),\bu(\tau))+(\bv(\tau), \bg(\tau))| \leq e^4 \tfrac t2 (|\mathbf{f}(\tau)|^2 +|\mathbf g(\tau)|^2)+ \tfrac1{t} E(t)
\]
and the Gronwall inequality we obtain the estimate
\[
E(t) \leq e \left(e^2 E(0) + e^4\frac t2 \int_0^t \left(|\mathbf{f}(\tau)|^2  + |\bg(\tau)|^2 \right)d\tau\right).
\]

(c) By linearity it remains to study the case $E(0) = 0$, $\mathbf{f} = 0$, and $\mathbf{g} = 0$.  We consider the Laplace transformed equations:
\begin{align*}
  s\widehat \bu &= -\mathbf{D}^T \widehat \bv - \mathbf{C}_0 \widehat\bvarphi\\
  s\widehat \bv &= \mathbf{D}\widehat \bu - \mathbf{C}_1 \widehat\bpsi\\
 \mathbf{B}(s)
 \begin{pmatrix}
   \widehat\bvarphi\\\widehat\bpsi
 \end{pmatrix}
&= 
\begin{pmatrix}
  \mathbf{C}_0^T \widehat \bu\\ \mathbf{C}_1^T \widehat \bv
\end{pmatrix}
+
\begin{pmatrix}
  \widehat  \brho\\  \widehat \bsigma
\end{pmatrix}.
\end{align*}
We take the inner product of the first equation with $ \widehat \bu$, the second with $ \widehat \bv$, and the third with $(\widehat \bvarphi, \widehat \bpsi)$ and sum up the real parts to obtain
\[
\Re s\,|\widehat \bu|^2 + \Re s\,|\widehat \bv|^2 +\Re\left \langle
\begin{pmatrix}
  \widehat \bvarphi\\\widehat \bpsi
\end{pmatrix}
, 
\mathbf{B}(s)
\begin{pmatrix}
  \widehat \bvarphi\\\widehat \bpsi
\end{pmatrix}
\right\rangle_\Gamma
= 
\Re\langle\widehat \bvarphi, \widehat \brho \rangle_\Gamma+ 
\Re \langle\widehat \bpsi,\widehat \bsigma \rangle_\Gamma.
\]
Using Lemma~\ref{lemma:calderon_pos} we obtain
\[
\beta \min(1,|s|^2)\frac{\Re s}{|s|^2} \left(|\widehat \bvarphi|^2+ |\widehat \bpsi|^2\right)
\leq |\langle\widehat \bvarphi, \widehat \brho \rangle_\Gamma|+ |\langle\widehat \bpsi,\widehat \bsigma \rangle_\Gamma|.
\]
Using the Cauchy-Schwarz inequality 
\[
|\langle\widehat \bvarphi, \widehat \brho \rangle_\Gamma| 
\leq \tfrac\beta{2} \min(1,|s|^2)\Re s|s^{-1}\widehat \bvarphi|^2+ 
\tfrac1{2\beta} \bigl(\min(1,|s|^2)\Re s\bigr)^{-1}|s\widehat \brho|^2
\]
 we obtain for $\Re s > 1/t$
\[
|\widehat \bvarphi|^2+ |\widehat \bpsi|^2 \leq \beta^{-2}\max(t^2,t^6)
\left(  |s^2 \widehat \brho|^2+ |s^2 \widehat \bsigma|^2\right).
\]
With the Plancherel formula and causality we conclude
\begin{equation}
  \label{eq:E_L2_phi}
\int_0^t \left(|\bvarphi(\tau)|^2 + |\bpsi(\tau)|^2\right)  d\tau
\leq (e/\beta)^2 \max(t^2,t^6)  \int_0^t  \left(|\ddot\brho(\tau)|^2 + |\ddot\bsigma(\tau)|^2 \right) d\tau.
\end{equation}
Note now that
\begin{equation}\label{eq:sec6a}
\begin{split}  
\int_0^t |\brho(\tau)|^2 d\tau &= \int_0^t \left|\int_0^\tau (\tau-\alpha)\ddot \brho(\alpha)\, d\alpha\right|^2d\tau \leq \int_0^t \frac{\tau^3}{3} \int_0^\tau |\ddot\brho(\alpha)|^2\,
d\alpha \, d\tau\\
&\leq \int_0^t \frac{\tau^3}{3} \int_0^t |\ddot\brho(\alpha)|^2\,d\alpha\, d\tau
= \frac{t^4}{12} \int_0^t |\ddot \brho(\tau) |^2 \, d\tau.
\end{split}
\end{equation}
Applying the Cauchy-Schwarz inequality in the last term of \eqref{eq:energy_ineq},  and using \eqref{eq:E_L2_phi} and \eqref{eq:sec6a}, proves the result if $E(0) = 0$, $\mathbf{f} = 0$, and $\bg = 0$.

(d)  Denoting the solution of (b) by $\bu_\Omega$ and that of part (c) by $\bu_\Gamma$, the solution of the general problem is given as $\bu = \bu_\Omega + \bu_\Gamma$ and bounded by $|\bu|^2 \leq 2(|\bu_\Omega|^2+|\bu_\Gamma|^2)$. With the estimates of (b) and (c) this gives the result.
\qed \end{proof}

\subsection{Mechanical energy}
Differentiating the first and last equations in \eqref{eq:semi_disc_perturb} and eliminating $\bv$ yields the second-order formulation
\begin{equation}
  \label{eq:ddot-u}  
  \begin{array}{rcl}    
\ddot \bu &=& -\mathbf{D}^T \mathbf{D}\bu + \mathbf{D}^T\mathbf{C}_1\bpsi - \mathbf{C}_0\dot \bvarphi+\dot{\mathbf{f}}-\mathbf{D}^T\bg
\\
 \mathbf{B}(\partial_t)  \begin{pmatrix}
  \dot \bvarphi\\ \dot \bpsi
  \end{pmatrix}
&=& 
\begin{pmatrix}
  \mathbf{C}_0^T \dot \bu \\ \mathbf{C}_1^T (\mathbf{D}\bu-\mathbf{C}_1\bpsi)
\end{pmatrix}
+
\begin{pmatrix}
 \dot \brho\\ \dot \bsigma + \mathbf{C}_1^T \bg 
\end{pmatrix}.
  \end{array}
\end{equation}
\begin{lemma}\label{lemma:semiH}
The semi-discrete mechanical energy
\[
H(t) = \half |\dot \bu(t)|^2+ \half | \mathbf{D}\bu(t) - \mathbf{C}_1 \bpsi(t)|^2
\]
is bounded along the solutions of \eqref{eq:semi_disc_perturb} by
\[
\begin{split}  
H(t) \leq   C(\beta) \Bigl(H(0)\Bigr. &+ t \int_0^t |\dot{\mathbf{f}}(\tau)-\mathbf{D}^T \bg(\tau)|^2 d\tau\\ &+  \Bigl.\max(t^2,t^6) \int_0^t \left( |\ddot \brho(\tau)|^2 + |\ddot \bsigma(\tau)+\mathbf C^T_1 \dot{\bg}(\tau)|^2 \right) d\tau \Bigr),
\end{split}
\]
for all $t > 0$.  This estimate holds provided $\brho(0) = \dot \brho(0) = 0$, $\bsigma(0) = \dot \bsigma(0) = 0$, and $\bg(0) = 0$.
\end{lemma}
\begin{proof}
  The proof is similar to that of Lemma~\ref{lemma:semiE}. 

(a) We take the inner product of the first equation in \eqref{eq:ddot-u} with $\dot \bu$ and the second with $(\dot \bvarphi, \dot \bpsi)^T$ and sum up:
\[
\dot H +\left\langle
  \begin{pmatrix}
    \dot \bvarphi\\\dot \bpsi
  \end{pmatrix}, \mathbf{B} (\partial_t) 
  \begin{pmatrix}
    \dot \bvarphi\\\dot\bpsi
  \end{pmatrix}
\right\rangle_\Gamma
= (\dot \bu, \dot{\mathbf{f}}-\mathbf{D}^T \bg) +\left\langle
\begin{pmatrix}
  \dot\bvarphi\\\dot\bpsi
\end{pmatrix}
,
\begin{pmatrix}
  \dot \brho\\ \dot \bsigma + \mathbf{C}_1^T \bg
\end{pmatrix}
\right \rangle_\Gamma.
\]
Integrating and using the positivity property of $\mathbf{B}(\partial_t)$ (Lemma \ref{lemma:4.2})  gives that the semi-discrete mechanical energy satisfies, for $t > 0$,
\begin{equation}
  \label{eq:Henergy_ineq}  
\begin{split}  
H(t)&+ \beta\min(t^{-1},t^{-3})\int_0^t
\Bigl( | \bvarphi (\tau)|^2 +|  \bpsi(\tau) |^2\Bigr)d\tau \\ \le 
& e^2 \Big( H(0)+ \int_0^t |(\dot \bu(\tau), \dot{\mathbf{f}}(\tau)-\mathbf{D}^T \bg(\tau))|d\tau\\
&  +\int_0^t |\langle\dot\bvarphi(\tau),\dot \brho(\tau)\rangle_\Gamma+\langle\dot\bpsi(\tau),\dot \bsigma(\tau) + \mathbf{C}_1^T \bg(\tau)\rangle_\Gamma |d\tau\Big).
\end{split}
\end{equation}
While $\dot{\mathbf{f}}-\mathbf{D}^T\bg$ poses no problems, the dependence on the boundary perturbations $\dot \brho$ and $\dot \bsigma+\mathbf C^T_1 \bg$ needs to be treated in a different way. 

(b) If we first assume that $\dot \brho$ and $\dot \bsigma+\mathbf C^T_1 \bg$ are zero, then using the Cauchy-Schwarz inequality  and the Gronwall inequality we obtain the estimate
\[
H(t) \leq e \left(e^2 H(0) + e^4\frac t2 \int_0^t|\dot{\mathbf{f}}(\tau)-\mathbf{D}^T\bg(\tau)|^2d\tau\right).
\]

(c) By linearity it remains to study the case $\bu(0) = \dot \bu(0) = 0$ and $\dot{\mathbf{f}}-\mathbf{D}^T\bg = 0$.  We consider the Laplace transformed equations:
\begin{align*}
  s^2\widehat \bu &= -\mathbf{D}^T \mathbf{D} \widehat \bu+\mathbf{D}^T\mathbf{C}_1 \widehat \bpsi -s \mathbf{C}_0 \widehat\bvarphi\\
 \mathbf{B}(s)
 \begin{pmatrix}
   s\widehat\bvarphi\\s\widehat\bpsi
 \end{pmatrix}
&= 
\begin{pmatrix}
  s\mathbf{C}_0^T \widehat \bu\\ \mathbf{C}_1^T (\mathbf{D} \widehat \bu-\mathbf{C}_1  \widehat \bpsi)
\end{pmatrix}
+
\begin{pmatrix}
  s\widehat  \brho\\  s\widehat \bsigma+\mathbf{C}_1^T\widehat \bg
\end{pmatrix}.
\end{align*}
We take the inner product of the first equation with $ s\widehat \bu$ and the second with $(s\widehat \bvarphi, s\widehat \bpsi)^T$ and sum up to obtain
\[
s|s\widehat \bu|^2 + \bar s|\mathbf{D}\widehat \bu - \mathbf{C}_1 \widehat \bpsi|^2 
+ \left\langle  \begin{pmatrix}
   s\widehat\bvarphi\\s\widehat\bpsi
 \end{pmatrix},
 \mathbf{B}(s)
 \begin{pmatrix}
   s\widehat\bvarphi\\s\widehat\bpsi
 \end{pmatrix}
\right\rangle_\Gamma
= \left\langle 
\begin{pmatrix}
   s\widehat\bvarphi\\s\widehat\bpsi
 \end{pmatrix}, 
\begin{pmatrix}
  s\widehat  \brho\\  s\widehat \bsigma+\mathbf{C}_1^T\widehat \bg
\end{pmatrix}
\right\rangle_\Gamma.
\]
Taking the real part, using the positivity of $\mathbf{B}(s)$ on the left-hand side  and the triangle inequality on the right-hand side we obtain
\[
\beta\min(1,|s|^{-2}) \Re s \left(|\widehat \bvarphi|^2+|\widehat \bpsi|^2\right)
\leq |\langle\widehat \bvarphi, s^2\widehat \brho \rangle_\Gamma|+ |\langle\widehat \bpsi, s^2\widehat \bsigma + s\mathbf{C}_1^T \widehat \bg\rangle_\Gamma|.
\]
Using the Cauchy-Schwarz inequality   we obtain for $\Re s > 1/t$
\[
|\widehat \bvarphi|^2+ |\widehat \bpsi|^2 \leq \beta^{-2}\max(t^2,t^6)
\left(  |s^2 \widehat \brho|^2+ |s^2 \widehat \bsigma+ s\mathbf{C}_1^T \widehat \bg|^2\right).
\]
With the Plancherel formula and causality we conclude
\[
\int_0^t \left(|\bvarphi(\tau)|^2 + |\bpsi(\tau)|^2\right)  d\tau
\leq (e/\beta)^2 \max(t^2,t^6)  \int_0^t  \left(|\ddot\brho(\tau)|^2 + |\ddot\bsigma(\tau)+\mathbf{C}_1^T \dot\bg(\tau)|^2 \right) d\tau.
\]
Using the Cauchy-Schwarz inequality in the last term of \eqref{eq:Henergy_ineq}
and
\[
\int_0^t |\brho(\tau)|^2 d\tau \leq \frac{t^4}{12} \int_0^t |\ddot \brho(\tau) |^2 d\tau
\]
gives the result if $H(0) = 0$ and $\dot{\mathbf{f}} -\mathbf{D}^T \mathbf g= 0$.

(d)  As in the previous proof we conclude to the stated result using linearity and the estimates in (b) and (c).
\qed \end{proof}

\subsection{Boundary functions}

\begin{lemma}\label{lemma:semiPhi}
The boundary functions of \eqref{eq:semi_disc_perturb} are bounded as
\[
\begin{split}  
\int_0^t (|\bvarphi(\tau)|^2 &+ |\bpsi(\tau)|^2)  d\tau
\\ &\leq C(\beta) \max(t^2,t^6)  \int_0^t  \left(|\dot{\mathbf{f}}(\tau)|^2+|\dot \bg(\tau)|^2+ |\ddot\brho(\tau)|^2 + |\ddot\bsigma(\tau)|^2 \right) d\tau,
\end{split}
\]  
for all $t > 0$. This estimate holds provided that $\mathbf{f}(0) = 0$, $\bg(0) = 0$, $\brho(0) = \dot \brho(0) = 0$, and $\bsigma(0) = \dot \bsigma(0) = 0$.
\end{lemma}
\begin{proof}
  We separate the three cases (i) $\bu(0) = 0$, $\bv(0) = 0$, $\mathbf{f} = 0$ and $\bg = 0$, (ii) $\dot \brho = 0$, $\dot \bsigma = 0$ and $\bg = 0$, and (iii) all inhomogeneities and initial values vanish except for an arbitrary $\bg$.

In the case (i) an estimate of the temporal $L_2$ norms of $\bvarphi$ and $\bpsi$ is given in \eqref{eq:E_L2_phi}. In the case (ii) such an estimate follows from \eqref{eq:Henergy_ineq}. It remains to study the case (iii), which is done by an extension of part (c) of the proof of Lemma~\ref{lemma:semiE}.

We consider the Laplace transformed equations:
\begin{align*}
  s\widehat \bu &= -\mathbf{D}^T \widehat \bv - \mathbf{C}_0 \widehat\bvarphi\\
  s\widehat \bv &= \mathbf{D}\widehat \bu - \mathbf{C}_1 \widehat\bpsi+ \widehat \bg\\
 \mathbf{B}(s)
 \begin{pmatrix}
   \widehat\bvarphi\\\widehat\bpsi
 \end{pmatrix}
&= 
\begin{pmatrix}
  \mathbf{C}_0^T \widehat \bu\\ \mathbf{C}_1^T \widehat \bv
\end{pmatrix}
\end{align*}
We take the inner product of the first equation with $ \widehat \bu$, the second with $ \widehat \bv$, and the third with $(\widehat \bvarphi, \widehat \bpsi)^T$ and sum up to obtain
\[
s|\widehat \bu|^2 + s|\widehat \bv|^2 +\left \langle
\begin{pmatrix}
  \widehat \bvarphi\\\widehat \bpsi
\end{pmatrix}
, 
\mathbf{B}(s)
\begin{pmatrix}
  \widehat \bvarphi\\\widehat \bpsi
\end{pmatrix}
\right\rangle_\Gamma
= 
(\widehat \bv, \widehat \bg).
\]
Taking the real part and using Lemma~\ref{lemma:calderon_pos} we obtain
\[
 \Re s|\widehat \bv|^2+\beta \min(1,|s|^2)\frac{\Re s}{|s|^2} \left(|\widehat \bvarphi|^2+ |\widehat \bpsi|^2\right)
\leq \Re s |\widehat \bv|^2 + \frac1{4\Re s}|\widehat \bg|^2.
\]
Hence  we obtain for $\Re s > 1/t$
\[
|\widehat \bvarphi|^2+ |\widehat \bpsi|^2 \leq \beta^{-1}\max(t^2,t^6)
  |s \widehat \bg|^2.
\]
With the Plancherel formula and causality we conclude
\[
\int_0^t \left(|\bvarphi(\tau)|^2 + |\bpsi(\tau)|^2\right)  d\tau
\leq \beta^{-1}\max(t^2,t^6)  \int_0^t  |\dot \bg(\tau)|^2 d\tau.
\]
Combining the cases (i)--(iii) gives the result.
\qed \end{proof}

\section{Error bound for the spatial semidiscretization}
\label{sec:err-semi}

\subsection{Consistency errors}

We denote by $P^U_h$ and $P^V_h$  the $L_2(\Omega)$-orthogonal projections onto the finite element spaces $U_h$ and $V_h$, respectively, and by $P^{\Phi}_h$ and $P^{\Psi}_h$ the $L_2(\Gamma)$-orthogonal projections onto the boundary element spaces $\Phi_h$ and  $\Psi_h$, respectively. We omit the superscripts $U, V, \Phi, \Psi$ when they are clear from the context.

We consider the defects obtained when we insert the projected exact solution $(P_hu,P_hv,P_h\varphi,P_h\psi)$ into the variational formulation.  We obtain
\begin{equation}\label{eq:defects}
\begin{aligned}
&  ( P_h\dot u, w) = -\half(P_h v,\nabla w) + \half (\nabla\cdot P_h v,w) -\half\langle P_h \varphi,\gamma w\rangle_\Gamma +(f,w)\\
& \qquad \qquad \quad + \half (\nabla \cdot (v- P_h v), w)\\
&  (P_h\dot v, z) = -\half(P_h u,\nabla\cdot z) + \half (\nabla P_h u,z) +\half\langle P_h \psi,\gamma z\cdot n\rangle_\Gamma\\
& \qquad \qquad \quad + \half (\nabla (u- P_h u), z)\\
& \left\langle
  \begin{pmatrix}
    \xi\\\eta
  \end{pmatrix}, {B} (\partial_t) 
  \begin{pmatrix}
    P_h \varphi\\P_h \psi
  \end{pmatrix}
\right\rangle_\Gamma =\half\left \langle
\begin{pmatrix}
\xi \\ \eta 
\end{pmatrix},
\begin{pmatrix}
  \gamma P_h u\\ -\gamma P_h v \cdot n
\end{pmatrix}\right\rangle_\Gamma
\\
 & \qquad \qquad \quad -
\left\langle
  \begin{pmatrix}
    \xi\\\eta
  \end{pmatrix}, B(\partial_t) 
  \begin{pmatrix}
    \varphi-P_h \varphi\\\psi-P_h \psi
  \end{pmatrix}
\right\rangle_\Gamma\\
 & \qquad \qquad \quad +
\half\left \langle
\begin{pmatrix}
\xi \\ \eta 
\end{pmatrix},
\begin{pmatrix}
  \gamma (u-P_h u)\\ -\gamma (v-P_h v) \cdot n
\end{pmatrix}\right\rangle_\Gamma.
\end{aligned}
\end{equation}
The defects are estimated using the following lemmas and the trace inequalities $\|\gamma w\|_{H^{1/2}(\Gamma)} \leq C \|w\|_{H^1(\Omega)}$ and
$\|\gamma z \cdot n\|_{H^{-1/2}(\Gamma)} \leq C \|z\|_{H^1(\Omega)}$.

\begin{lemma}\label{lemma:PhH1}
  In the case of a quasi-uniform triangulation of $\Omega$, there exists a positive constant $C$ such that
\[
\|w-P_hw\|_{H^1(\Omega)} \leq C h |w|_{H^2(\Omega)}\ \quad
\text{for all } w \in H^2(\Omega).
\]
\end{lemma}
\begin{proof}
  We denote by $I_h$ the finite element interpolation operator and write
\[
w-P_h w = (w-I_h w) + (I_hw-P_hw).
\]
The $H^1(\Omega)$ norm of the first term is of $O(h)$ by standard finite element theory.  The $L_2(\Omega)$ norm of the second term is $O(h^2)$ and hence the result follows using an inverse inequality.
\qed \end{proof}

\begin{lemma}\label{lemma:Bdefect}
  There exists a constant $C(t)$ growing at most polynomially with $t$ such that
\[
\begin{split}  
\int_0^t &\left\| B(\partial_t)
  \begin{pmatrix}
    (I-P_h)\varphi(\cdot, \tau)\\
    (I-P_h)\psi(\cdot, \tau)
  \end{pmatrix}
\right\|^2_{H^{1/2}(\Gamma) \times H^{-1/2}(\Gamma)}d\tau\\
&\leq C(t) h^2 \int_0^t \left(
\|\partial_t^2 \varphi(\cdot, \tau)\|^2_{H^{1/2}(\Gamma)}
+ \|\partial_t^2 \psi(\cdot, \tau)\|^2_{H^{3/2}(\Gamma)}
\right) d\tau,
\end{split}
\]
for any $t > 0$ and for all $\varphi \in C^2([0,t], H^{1/2}(\Gamma))$, $\psi \in C^2([0,t], H^{3/2}(\Gamma))$ with $\varphi(\cdot, 0) = \partial_t\varphi(\cdot, 0) = 0$ and $\psi(\cdot, 0) = \partial_t\psi(\cdot, 0) = 0$.
\end{lemma}
\begin{proof}
  We first investigate the action of the blocks of $B(s)$ on the projection errors.  By the bounds given in Section~\ref{section:cald_Helm} and by the standard approximation estimates for boundary element spaces we obtain for $\Re s \geq \sigma > 0$
\begin{align*}  
\|s V(s) (I-P_h) \varphi\|_{H^{1/2}(\Gamma)} &\leq C(\sigma) |s|^2 \|\varphi-P_h \varphi\|_{H^{-1/2}(\Gamma)} \\ &\leq C|s|^2 h \|\varphi\|_{H^{1/2}(\Gamma)}.
\end{align*} 
Similar bounds hold for the other blocks, so that
\[
\left\|B(s) 
  \begin{pmatrix}
    (I-P_h)\varphi\\
    (I-P_h)\psi
  \end{pmatrix}
\right\|_{H^{1/2}(\Gamma) \times H^{-1/2}(\Gamma)}
\leq C(\sigma) |s|^2 h \left( \|\varphi\|_{H^{1/2}(\Gamma)}+ \|\psi\|_{H^{3/2}(\Gamma)} \right).
\]
The result now follows by Plancherel's formula and causality.
\qed \end{proof}

With the above two lemmas, the consistency errors have been estimated.
\subsection{Error bound}
Combining the previous lemmas we obtain the following result.

\begin{theorem}
  Assume that the initial values $u(\cdot, 0)$ and $v(\cdot,0)$ have their support in $\Omega$. Let the initial values for the semi-discretization be chosen as $u_h(0) = P_h u(0)$ and $v_h(0) = P_h v (0)$, where $P_h$ denotes the $L_2(\Omega)$-orthogonal projection onto the finite element spaces. If we assume that the solution of the wave equation  \eqref{eq:wave} is sufficiently smooth, then the error of the FEM-BEM semi-discretization \eqref{eq:semi_disc} is bounded by
\[
\begin{split}  
&\|u_h(t)-u(t)\|_{L_2(\Omega)} + \|v_h(t)-v(t)\|_{L_2(\Omega)^3}\\ &+
\left(\int_0^t \|\varphi_h(\tau)-\varphi(\tau)\|^2_{H^{-1/2}(\Gamma)}+ \|\psi_h(\tau)-\psi(\tau)\|^2_{H^{1/2}(\Gamma)}d\tau\right)^{1/2}
\leq C(t) h,
\end{split}
\]
where the constant $C(t)$ grows at most polynomially with $t$.
\end{theorem}
\begin{proof}
We apply the stability lemmas to the differences $u_h-P_hu$, $v_h-P_h v$, $\varphi_h-P_h \varphi$, and $\psi_h - P_h\psi$ and denote the defects in \eqref{eq:defects} by
\[
f_h = \half \nabla \cdot (v-P_h v), \qquad
g_h = \half \nabla  (u-P_h u)
\]
and
\[
\begin{pmatrix}
  \rho_h\\\sigma_h
\end{pmatrix}
= - B(\partial_t) 
\begin{pmatrix}
  \varphi - P_h \varphi\\
\psi - P_h \psi
\end{pmatrix}
+ \half
\begin{pmatrix}
  \gamma (u-P_h u)\\ -\gamma(v-P_h v) \cdot n
\end{pmatrix}.
\]
Translating Lemma~\ref{lemma:semiE} into the functional analytic setting gives the estimate 
\[
\begin{split}
  \|u_h(t)&-P_h u(t)\|_{L_2(\Omega)}^2 +   \|v_h(t)-P_h v(t)\|_{L_2(\Omega)^3}^2
\\ &\leq C(\beta)\Big( t \int_0^t \left( \|f_h(\cdot,\tau)\|^2_{L_2(\Omega)}+\|g_h(\cdot,\tau)\|^2_{L_2(\Omega)^3}\right) d\tau \\ &+ \max(t^2,t^6) \int_0^t \left(
\|\ddot \rho_h(\cdot, \tau)\|^2_{H^{1/2}(\Gamma)}+\|\ddot \sigma_h(\cdot, \tau)\|^2_{H^{-1/2}(\Gamma)}
\right)d\tau\Big).
\end{split}
\]
Similarly Lemma~\ref{lemma:semiPhi} translates into
\[
\begin{split}  
\int_0^t& \left(\|\varphi_h(\cdot, \tau)-P_h\varphi(\cdot,\tau)\|^2_{H^{-1/2}(\Gamma)} + \|\psi_h(\cdot, \tau)-P_h\psi(\cdot,\tau)\|^2_{H^{1/2}(\Gamma)} \right)d\tau
\\ \leq & C(\beta) \max(t^2,t^6)\Bigg( \int_0^t \left(\|\partial_t f_h(\cdot,\tau)\|^2_{L_2(\Omega)} + \|\partial_t g_h(\cdot,\tau)\|^2_{L_2(\Omega)^3}\right)d\tau
\\ &+\int_0^t \left( \|\partial^2_t \rho_h(\cdot,\tau)\|^2_{H^{1/2}(\Gamma)} + \|\partial^2_t \sigma_h(\cdot,\tau)\|^2_{H^{-1/2}(\Gamma)}\right)d\tau\Bigg).
\end{split}
\]
The conditions on the vanishing initial values required in Lemma~\ref{lemma:semiE} and Lemma~\ref{lemma:semiPhi} are satisfied because we assumed that the initial data of the wave equation have their support in $\Omega$ and because we chose the initial values of the space discretization as the appropriate projections of the initial data.

Using  the estimates of Lemma~\ref{lemma:PhH1} and Lemma~\ref{lemma:Bdefect} yields the result.
\qed \end{proof}

We remark that higher-degree finite elements and boundary elements yield correspondingly higher order, provided that the solution is sufficiently smooth.

\section{Stability of the full discretization}
\label{sec:stab-disc}

\subsection{Setting of the stability analysis}
\label{subsec:stab-discrete-setting}
In this section we study the stability of the fully discrete scheme  under the CFL condition
\begin{equation}\label{cfl}
\dt \|\mathbf D\| \le 1
\end{equation}
and the lower bound on the stabilization parameter
\begin{equation}\label{alpha}
\alpha \ge 1.
\end{equation}
We remark that the same kind of results can be obtained under the weaker CFL bound $\dt \|D\| \le \rho < 2$ for sufficiently large $\alpha$. The lower bound on $\alpha$ tends to infinity as $\rho\to 2$.

We consider the setting of Section~\ref{subsec:stab-semidiscrete-setting} and bound the Euclidean norms of the solutions of the perturbed discrete scheme
\begin{eqnarray*}
 \bv^{n+1/2} &=&  \bv^n +\half\dt \,\mathbf{D}u^n - \half \dt \,\mathbf{C}_1\bpsi^n +\half\dt\, \mathbf{g}^n
\\
 \bu^{n+1} &=& \bu^n - \dt \,\mathbf{D}^T \bv^{n+1/2} - \dt\, \mathbf{C}_0\bvarphi^{n+1/2}+\dt  \,\mathbf{f}^{n+1/2}
\\
\bv^{n+1} &=& \bv^{n+1/2} +\half\dt \,\mathbf{D}\bu^{n+1} - \half \dt \,\mathbf{C}_1\bpsi^{n+1} +\half\dt\, \mathbf{g}^{n+1}
\end{eqnarray*}
and
\begin{eqnarray*}
\biggl[  \mathbf{B}(\partial_t^\dt)  \begin{pmatrix}
   \bvarphi\\ \bar\bpsi
  \end{pmatrix}\biggr]^{n+1/2}
&=& 
\begin{pmatrix}
  \mathbf{C}_0^T \bar \bu^{n+1/2} \\ \mathbf{C}_1^T (\bv^{n+1/2}-\alpha \dt^2 \mathbf{C}_1\dot\bpsi^{n+1/2})
\end{pmatrix} + 
\begin{pmatrix}
\brho^{n+1/2}
\\
\bsigma^{n+1/2}
\end{pmatrix}
,
\end{eqnarray*}
where again $ \bar \bu^{n+1/2}= \half (\bu^{n+1}+\bu^n)$ and $ \bar \bpsi^{n+1/2}= \half (\bpsi^{n+1}+\bpsi^n)$, and $\dot\bpsi^{n+1/2}=(\bpsi^{n+1}-\bpsi^n)/\dt$.

We will proceed in parallel to Section~\ref{sec:stab-semidiscrete} and transfer the arguments from the semidiscrete to the discrete situation, concentrating on the extra difficulties.

\subsection{Discrete field energy}

\begin{lemma}\label{lemma:discrete-E}
Under conditions (\ref{cfl}) and (\ref{alpha}), the discrete field energy
\[
E^n=\half |\bu^n|^2 + \tfrac14 \bigl( |\bv^{n+1/2}|^2 + |\bv^{n-1/2}|^2 \bigr) 
\]
is bounded, at $t=n\dt$, by
\[
\begin{split}  
E^n \leq  C \Bigl(E^0\Bigr. &+ \frac t2 \, \dt\sum_{j=0}^n \left(|\mathbf{f}^{j+1/2}|^2  + |\bg^j|^2 \right)\\ &+  \Bigl.\max(t^2,t^6) \dt \sum_{j=0}^n \left( |(\partial_t^\dt)^2 \brho^{j+1/2}|^2 + |(\partial_t^\dt)^2 \bsigma^{j+1/2}|^2 \right)  \Bigr),
\end{split}
\]
where $C$ is independent of $h$, $\dt$, and $n$.
\end{lemma}

Since $\bv^n=\tfrac12(\bv^{n+1/2}+\bv^{n-1/2})$, this result also yields a bound of $|\bv^n|^2$ of the same type.

\begin{proof} (a) The recursion for $\bv$ is conveniently expressed in the midpoint values $\bv^{n+1/2}$ only:
$$
\bv^{n+1/2} = \bv^{n-1/2} + \dt \,\mathbf{D}\bu^n -  \dt \,\mathbf{C}_1\bpsi^n +  \dt \, \bg^n.
$$
We take the inner product with 
$\half \bar  \bv^n = \tfrac14 (\bv^{n+1/2} + \bv^{n-1/2})=\half \bv^n$ in this equation, with $\bar \bu^{n+1/2}$ in the equation for $\bu^{n+1}$, 
with half times $\bar \bv^{n+1}$ in the equation for $\bv^{n+1}$, and with
$(\bvarphi^{n+1/2},\bar\bpsi^{n+1/2})$ in the boundary equation. We sum up the resulting four equations to obtain
\begin{eqnarray*}
&&\tfrac14 |\bv^{n+3/2}|^2 -\tfrac14 |\bv^{n-1/2}|^2 +\half |\bu^{n+1}|^2 - \half |\bu^{n}|^2 
\\
&&-\ \half \dt (\bar \bv^n, \mathbf{D}\bu^n-\mathbf{C}_1\bpsi^n) - \half \dt (\bar \bv^{n+1}, \mathbf{D}\bu^{n+1}-\mathbf{C}_1\bpsi^{n+1})
\\
&& 
\quad +\: \dt (\bv^{n+1/2}, \mathbf{D} \bar \bu^{n+1/2}-\mathbf{C}_1\bar\bpsi^{n+1/2}) 
+\alpha\dt^3(\mathbf{C}_1\dot\bpsi^{n+1/2},\mathbf{C}_1\bar\bpsi^{n+1/2})
\\
 && +\  \dt
 \left\langle 
  \begin{pmatrix}
   \bvarphi^{n+1/2}\\ \bar\bpsi^{n+1/2}
  \end{pmatrix},
 \biggl[  \mathbf{B}(\partial_t^\dt)  
 \begin{pmatrix}
   \bvarphi\\ \bar\bpsi
  \end{pmatrix}
  \biggr]^{n+1/2}
\right\rangle
 \\
 && = \half\dt (\bar \bv^n, \bg^n) +  \half\dt (\bar \bv^{n+1}, \bg^{n+1}) + \dt (\bar \bu^{n+1/2},\mathbf{f}^{n+1/2} ) 
 \\
 &&+\ \dt \langle \bvarphi^{n+1/2},\brho^{n+1/2}\rangle + 
  \Delta t   \langle \bsigma^{n+1/2}, \bar\bpsi^{n+1/2} \rangle.
\end{eqnarray*}
Here we note that on setting 
$
\dot \bv^n = (\bv^{n+1/2}-\bv^{n-1/2})/{\dt}
$
we have
\begin{eqnarray*}
&&\half \dt (\bar \bv^n, \mathbf{D}\bu^n-\mathbf{C}_1\bpsi^n) + \half \dt (\bar \bv^{n+1}, \mathbf{D}\bu^{n+1}-\mathbf{C}_1\bpsi^{n+1})
\\
&& 
\qquad -\: \dt (\bv^{n+1/2}, \mathbf{D} \bar \bu^{n+1/2}-\mathbf{C}_1\bar\bpsi^{n+1/2}) 
\\
&& =
\tfrac14 \dt^2 (\dot \bv^{n+1},\mathbf{D}\bu^{n+1}-\mathbf{C}_1\bpsi^{n+1}) - \tfrac14 \dt^2 (\dot \bv^{n},\mathbf{D}\bu^{n}-\mathbf{C}_1\bpsi^{n}) 
\end{eqnarray*}
and
$$
\dt (\mathbf{C}_1\dot\bpsi^{n+1/2},\mathbf{C}_1\bar\bpsi^{n+1/2})= \half|\mathbf{C}_1\bpsi^{n+1}|^2 - \half |\mathbf{C}_1\bpsi^n|^2.
$$
Hence the first three lines in the above equation can be written as $\widetilde E^{n+1}-\widetilde E^n$
with the modified discrete field energy
\begin{align*}
\widetilde E^n =& 
\half |\bu^n|^2 + \tfrac14 \bigl( |\bv^{n+1/2}|^2 + |\bv^{n-1/2}|^2 \bigr) 
\\
& - 
\tfrac14 \dt^2 (\dot \bv^{n},\mathbf{D}\bu^{n}-\mathbf{C}_1\bpsi^{n})+\alpha \dt^2 \half |\mathbf{C}_1\bpsi^n|^2.
\end{align*}
Under the CFL condition (\ref{cfl}) we obtain by estimating
\begin{eqnarray*}
&& (\dot \bv^n, \mathbf{D}\bu^n-\mathbf{C}_1\bpsi^n) =|\mathbf{D}\bu^n-\mathbf{C}_1\bpsi^n|^2 +
 ( \bg^n, \mathbf{D}\bu^n-\mathbf{C}_1\bpsi^n) 
 \\
 && \quad\le 2 |\mathbf{D}\bu^n|^2 + 2|\mathbf{C}_1\bpsi^n|^2 + \tfrac1{2}|\bg^n|^2
\end{eqnarray*} 
that the modified discrete energy is bounded from below by
\begin{equation}\label{Emod-E}
\widetilde E^n \ge\tfrac14 |\bu^n|^2 + \tfrac14 \bigl( |\bv^{n+1/2}|^2 + |\bv^{n-1/2}|^2 \bigr) +\half (\alpha-1)\dt^2 |\mathbf{C}_1\bpsi^n|^2-
\tfrac18 {\dt^2}|\bg^n|^2.
\end{equation}
Note that the term with $\bpsi^n$ is non-negative for 
$\alpha\ge 1$.

We sum from $n=0$ to $m$ and note that by the positivity property of 
$\mathbf{B}(s)$ from Lemma~\ref{lemma:calderon_pos} and by 
Lemma~\ref{lemma:cq_Herglotz}, for $m\dt\le T$,
\begin{eqnarray*}
\sum_{n=0}^m
&&\left\langle 
 \begin{pmatrix}
   \bvarphi^{n+1/2}\\ \bar\bpsi^{n+1/2}
  \end{pmatrix},
 \biggl[  \mathbf{B}(\partial_t^\dt)  
 \begin{pmatrix}
   \bvarphi\\ \bar\bpsi
  \end{pmatrix}
  \biggr]^{n+1/2}
\right\rangle
\\
&&\ge \frac \beta{2eT} 
\dt \sum_{n=0}^{m-1} \Bigl( |(\partial_t^\dt)^{-1}\bvarphi^{n+1/2}|^2 + |(\partial_t^\dt)^{-1}\bar\bpsi^{n+1/2}|^2 \Bigr) .
\end{eqnarray*}
We then have
\begin{eqnarray}
\nonumber
\widetilde E^m- \widetilde E_0 &+&\frac \beta{2eT} 
\dt \sum_{n=0}^{m-1} \Bigl( |(\partial_t^\dt)^{-1}\bvarphi^{n+1/2}|^2 + |(\partial_t^\dt)^{-1}\bar\bpsi^{n+1/2}|^2 \Bigr)
\\
\label{E-mod-est}
&\le& \dt \sum_{n=0}^m \! '' \, |\bar \bv^{n}|\cdot |\bg^n| +\dt  \sum_{n=0}^{m-1} |\bar \bu^{n+1/2}|\cdot |\mathbf{f}^{n+1/2}| 
\\   \nonumber
&&+\  
\dt \sum_{n=0}^{m-1} \bigl(| \bvarphi^{n+1/2}|\cdot |\brho^{n+1/2}| +
 |\bar \bpsi^{n+1/2}|\cdot |\bsigma^{n+1/2}|\bigr),
\end{eqnarray}
where the double prime on the first sum indicates that the first and last term are taken with the factor $\half$. 

(b) If we first assume that all $\brho^n$ and $\bsigma^n$ are zero, then using the Cauchy-Schwarz inequality and Young's inequality,
and finally the discrete Gronwall inequality, we obtain the estimate at $t=n\dt$,
\[
\widetilde E^n \leq e \left(\widetilde E^0+ \frac t2 \, \dt\sum_{j=0}^n \left(|\mathbf{f}^{j+1/2}|^2  + |\bg^{j}|^2 \right)\right).
\]

(c) By linearity it remains to study the case $E^0 = 0$ and all $\mathbf{f}^{n+1/2} = 0$ and $\mathbf{g}^n = 0$.  We consider the  equations for the generating power series
$$
\widehat \bu(\zeta) = \sum_{n=0}^\infty \bu^n \zeta^n, \quad
\widehat \bv(\zeta) = \sum_{n=0}^\infty \bv^{n+1/2} \zeta^n,
$$
where $n$ is an exponent only on $\zeta$ and a time superscript else. We have,
  omitting the argument $\zeta$ in $\widehat \bu$, $\widehat \bv$, etc., and letting
  $s=\delta(\zeta)/\dt$ for brevity,
\begin{eqnarray*}
\frac{\zeta^{-1}-1}\dt \, \widehat \bu &=& -\mathbf{D}^T \widehat \bv - \mathbf{C}_0 \widehat \bvarphi
\\
\frac{1-\zeta}\dt\, \widehat \bv &=& \mathbf{D}\widehat \bu - \mathbf{C}_1 \widehat\bpsi
\\
\mathbf{B}(s) 
\begin{pmatrix}
\widehat\bvarphi \\ \widehat {\bar\bpsi}
\end{pmatrix}
&=& 
\begin{pmatrix}
 \mathbf{C}_0 ^T \widehat{\bar \bu} \\
 \mathbf{C}_1 ^T \widehat \bv - \alpha \dt^2 \mathbf{C}_1 ^T\mathbf{C}_1 \frac{\zeta^{-1}-1}\dt\,\widehat\bpsi
\end{pmatrix}
+ 
\begin{pmatrix}
\widehat \brho \\ \widehat \bsigma
\end{pmatrix}.
\end{eqnarray*}
where
$$
\widehat{\bar \bu} =\half(\zeta^{-1}+1)\widehat{\bu} , \qquad
\widehat{\bar \bpsi} =\half(\zeta^{-1}+1)\widehat{\bpsi} .
$$
We now use the energy method on the system for the generating power series. We take the inner product with $\widehat{\bar \bu}$ in the first equation, with 
$\half(1+\zeta)\widehat \bv$ in the second equation, and with $\begin{pmatrix}
\widehat\bvarphi \\ \widehat{\bar \bpsi}
\end{pmatrix}$ in the third equation. We sum up and take the real part to obtain
\begin{eqnarray*}
&&
\frac{|\zeta^{-1}|^2-1}{2\dt}\, |\widehat \bu|^2 + \frac{1-|\zeta|^2}{2\dt} |\widehat \bv|^2 + \Re \left\langle 
\begin{pmatrix}
\widehat\bvarphi \\   \widehat{\bar \bpsi}
\end{pmatrix},\mathbf{B}(s) 
\begin{pmatrix}
\widehat\bvarphi \\ \widehat{\bar \bpsi}
\end{pmatrix}
\right\rangle_\Gamma
\\
&& 
+\half \Re \!\Bigl( (  \zeta^{-1} - \overline \zeta) \, \widehat\bv^*(\mathbf{D} \widehat\bu-\mathbf{C}_1 \widehat\bpsi)\Bigr) +
\half \alpha \dt (|\zeta^{-1}|^2-1) |\mathbf{C}_1 \widehat\bpsi|^2
\\
&& 
=\  \Re \langle \widehat\bvarphi, \widehat\brho \rangle_\Gamma +
\Re \langle\widehat{\bar \bpsi}, \widehat\bsigma \rangle_\Gamma.
\end{eqnarray*}
Using the equation for $\widehat \bv$, the first term in the second line can be rewritten as
$$
\frac{|\zeta^{-1}|^2-1}{2\dt}\,\Re \frac { \overline \zeta}{1-\overline\zeta}\, 
\dt^2\, |\mathbf{D} \widehat\bu-\mathbf{C}_1 \widehat\bpsi|^2 .
$$
Here we note that 
$$
\Re \frac { \overline \zeta}{1-\overline\zeta} \ge -\frac12, \qquad |\zeta|<1,
$$
and under condition (\ref{cfl}),
$$
\half \dt^2\, |\mathbf{D} \widehat\bu-\mathbf{C}_1 \widehat\bpsi|^2 \le
|\widehat\bu|^2 + \dt^2\, |\mathbf{C}_1 \widehat\bpsi|^2.
$$
With condition (\ref{alpha}) we thus obtain, for $|\zeta|<1$,
$$
\Re \left\langle 
\begin{pmatrix}
\widehat\bvarphi \\ \widehat{\bar \bpsi}
\end{pmatrix},\mathbf{B}(s) 
\begin{pmatrix}
\widehat\bvarphi \\ \widehat{\bar \bpsi}
\end{pmatrix}
\right\rangle_\Gamma
\le  \Re \langle \widehat\bvarphi, \widehat\brho \rangle_\Gamma +
\Re \langle \widehat{\bar \bpsi}, \widehat\bsigma \rangle_\Gamma,
$$
and Lemma~\ref{lemma:calderon_pos} gives us
$$
\beta\, \min(1,|s|^2) \frac{\Re s}{|s|^2} \left(|\widehat\bvarphi|^2+ | \widehat{\bar \bpsi}|^2\right) \le 
 | \langle \widehat\bvarphi, \widehat\brho \rangle_\Gamma |+
| \langle\widehat{\bar \bpsi}, \widehat\bsigma \rangle_\Gamma|.
$$
Using the Cauchy-Schwarz inequality 
\[
|\langle\widehat \bvarphi, \widehat \brho \rangle_\Gamma| 
\leq \tfrac\beta{2} \min(1,|s|^2)\Re s|s^{-1}\widehat \bvarphi|^2+ 
\tfrac1{2\beta} \bigl(\min(1,|s|^2)\Re s\bigr)^{-1}|s\widehat \brho|^2
\]
 we obtain for $\Re s \ge 1/t$
\[
|\widehat \bvarphi|^2+ | \widehat{\bar \bpsi}|^2 \leq \beta^{-2}\max(t^2,t^6)
\left(  |s^2 \widehat \brho|^2+ |s^2 \widehat \bsigma|^2\right).
\]
For $s=\delta(\zeta)/\dt$ we have $\Re s \ge 1/t$ if $|\zeta|=\rho$ with $\rho=e^{-\mu\dt}$ for a $\mu=1/t + O(\dt)$.
With the Parseval formula on the circle $|\zeta|=\rho$ and causality we conclude, at $t=n\dt$,
\begin{eqnarray}\label{eq:E_L2_phi-disc}
&& \sum_{j=0}^n \left( |\bvarphi^{j+1/2}|^2 +  |\bar\bpsi^{j+1/2}|^2 \right)
\\ \nonumber
&& \leq 2(e/\beta)^2 \max(t^2,t^6) \sum_{j=0}^n \left(|(\partial_t^\dt)^2\brho^{j+1/2}|^2 + |(\partial_t^\dt)^2\bsigma^{j+1/2}|^2 \right) .
\end{eqnarray}
We now return to the bound (\ref{E-mod-est}), where we use a Cauchy-Schwarz inequality on the right-hand side and insert the above bound to obtain
$$
\widetilde E^n \leq  C \max(t^2,t^6) \dt \sum_{j=0}^n \left( |(\partial_t^\dt)^2 \brho^{j+1/2}|^2 + |(\partial_t^\dt)^2 \bsigma^{j+1/2}|^2 \right)  
$$
(d) By linearity, combining the estimates of (b) and (c) and recalling (\ref{Emod-E}) gives the stated result.
\qed\end{proof}

\subsection{Discrete mechanical energy}
In the following we denote $\dot \bu^{n+1/2}=(\bu^{n+1}-\bu^n)/\dt$, $\dot{\mathbf{f}}^n=
(\mathbf{f}^{n+1/2}-\mathbf{f}^{n-1/2})/\dt$, etc., and as previously,
$\bar \bu^{n+1/2}=\tfrac12 (\bu^{n+1}+\bu^n)$, $\bar \bpsi^{n+1/2}=\tfrac12 (\bpsi^{n+1}+\bpsi^n)$.

\begin{lemma}
The discrete mechanical energy
\[
H^{n+1/2}= \tfrac12 |\dot \bu^{n+1/2}|^2  +
\half |\mathbf{D}\bar\bu^{n+1/2}-\mathbf{C}_1\bar\bpsi^{n+1/2}|^2 
\]
is bounded at $t=(n+1/2)\dt$ by
\[
\begin{split}  
&H^{n+1/2}\leq  C \biggl(H^{1/2}\Bigr. + \frac t2 \sum_{j=0}^n |\dot{\mathbf{f}}^j-\mathbf{D}^T \bg^j|^2 
\\ 
&\quad + \Bigl.\max(t^2,t^6)\sum_{j=0}^n \left( |(\partial_t^\dt)^2 \dot\brho^{j}|^2 + |(\partial_t^\dt)^2 (\dot\bsigma+\mathbf{C}^T_1 \bg)^{j}|^2 \right) \biggr),
\end{split}
\]
where $C$ is independent of $h$, $\dt$, and $n$.
\end{lemma}

\begin{proof} (a)
We use a reformulation of the method. Like in the passage from the first-order formulation to the second-order formulation in the temporally continuous case,
we eliminate the variables $\bv$ in the equation. This gives us
$$
\bu^{n+1} - 2\bu^n + \bu^{n-1} = -\dt^2 \mathbf{D}^T(\mathbf{D}\bu^n-  \mathbf{C}_1\bpsi^n) - \dt^2 \mathbf{C}_0 \dot\bvarphi^n -
\dt^2 \mathbf{D}^T\bg^n +\dt^2  \mathbf{\dot f}^n.
$$
 Differencing the boundary equation yields, with $\dot{\bar\bpsi}^{n}=(\bar\bpsi^{n+1/2}-\bar\bpsi^{n-1/2})/\dt=(\bpsi^{n+1}-\bpsi^{n-1})/(2\dt)$ and
$\dot{\bar \bu}^{n}=(\bu^{n+1}-\bu^{n-1})/(2\dt)$, 
\begin{eqnarray*}
\biggl[  \mathbf{B}(\partial_t^\dt)  \begin{pmatrix}
   \dot\bvarphi\\ \dot{\bar\bpsi}
  \end{pmatrix}\biggr]^{n}
&=& 
\begin{pmatrix}
  \mathbf{C}_0^T \dot{\bar \bu}^{n} \\ \mathbf{C}_1^T (\mathbf{D}\bu^n-\mathbf{C}_1\bpsi^n)-\alpha\dt^2\mathbf{C}_1^T\mathbf{C}_1\ddot{\bpsi}^n
\end{pmatrix} 
+ 
\begin{pmatrix}
  \dot\brho^{n} \\ \mathbf{C}_1^T\bg^n+\dot\bsigma^{n} 
\end{pmatrix} 
\end{eqnarray*}
with $\ddot{\bpsi}^n=(\bpsi^{n+1}-2\bpsi^n+\bpsi^{n-1})/\dt^2$, and $\dot\brho^n=(\brho^{n+1/2}-\brho^{n-1/2})/\dt$ and
$\dot\bsigma^n=(\bsigma^{n+1/2}-\bsigma^{n-1/2})/\dt$.
%
We note that 
$$
\half(\dot \bu^{n+1/2} + \dot \bu^{n-1/2})=(\bar \bu^{n+1/2}-\bar \bu^{n-1/2})/\dt=(\bu^{n+1}-\bu^{n-1})/(2\dt)
$$
and hence
$
\bar {\dot \bu}^n = \dot{\bar \bu}^n .
$
We take the inner product with $\bar {\dot \bu}^n$ in the interior equation,  and with
$\half\dt(\dot\bvarphi^{n},\dot{\bar\bpsi}^{n})$ in the boundary equation. We note
\begin{eqnarray*}
(\dot{\bar\bu}^n, \ddot\bu^n) &=&
\frac1{2\dt}(\dot\bu^{n+1/2} + \dot\bu^{n-1/2},\dot\bu^{n+1/2} - \dot\bu^{n-1/2})
\\
&=&
\frac1{2\dt} \bigl( |\dot\bu^{n+1/2}|^2 - |\dot\bu^{n-1/2}|^2 \bigr),
\end{eqnarray*}
\begin{eqnarray*}
&&(\mathbf{D}\dot{\bar\bu}^n - \mathbf{C}_1\dot{\bar\bpsi}^n, \mathbf{D}{\bu}^n - \mathbf{C}_1{\bpsi}^n) 
\\
&&= \frac1\dt (\mathbf{D}{\bu}^{n+1} - \mathbf{C}_1{\bpsi}^{n+1},\mathbf{D}{\bu}^n 
- \mathbf{C}_1{\bpsi}^n)
\\
&&\qquad  - \frac1\dt (\mathbf{D}{\bu}^{n} - \mathbf{C}_1{\bpsi}^{n},\mathbf{D}{\bu}^{n-1}- \mathbf{C}_1{\bpsi}^{n-1}),
\end{eqnarray*}
and
$$
\langle \dot{\bar\bpsi}^n, \mathbf{C}_1^T\mathbf{C}_1 \ddot\bpsi^n\rangle =
\frac1\dt \bigl( |\mathbf{C}_1\dot\bpsi^{n+1/2}|^2 - |\mathbf{C}_1\dot\bpsi^{n-1/2}|^2).
$$
Summing all up and setting
$$
\widetilde H^{n+1/2} = \tfrac12 |\dot\bu^{n+1/2}|^2 + 
(\mathbf{D}{\bu}^{n+1} - \mathbf{C}_1{\bpsi}^{n+1},
\mathbf{D}{\bu}^n - \mathbf{C}_1{\bpsi}^n)  +
\alpha \dt^2  |\mathbf{C}_1\dot\bpsi^{n+1/2}|^2
$$
we obtain
\begin{eqnarray*}
&&\widetilde H^{n+1/2} - \widetilde H^{n-1/2} 
+\dt
 \left\langle 
  \begin{pmatrix}
   \dot\bvarphi^{n}\\ \dot{\bar\bpsi}^{n}
  \end{pmatrix},
 \biggl[  \mathbf{B}(\partial_t^\dt)  
 \begin{pmatrix}
   \dot\bvarphi\\ \dot{\bar\bpsi}
     \end{pmatrix}
  \biggr]^{n}
\right\rangle
\\
&&= \dt (\bar {\dot \bu}^n,\dot{\mathbf{f}}^n- \mathbf{D}^T\bg^n)  +
\dt \langle \dot\bvarphi^{n},\dot\brho^{n}\rangle + 
\dt \langle \dot{\bar\bpsi}^{n} ,  \mathbf{C}_1^T\bg^{n+1} +\dot\bsigma^{n}\rangle .
\end{eqnarray*}
Under the CFL condition (\ref{cfl}) we estimate, according to the formula $ab=\tfrac14(a+b)^2-\tfrac14(a-b)^2$, 
\begin{eqnarray*}
&&
(\mathbf{D}{\bu}^{n+1} - \mathbf{C}_1{\bpsi}^{n+1},
\mathbf{D}{\bu}^n - \mathbf{C}_1{\bpsi}^n) 
\\
&& = |\mathbf{D}{\bar\bu}^{n+1/2} - \mathbf{C}_1{\bar\bpsi}^{n+1/2}|^2 -
\frac{\dt^2}4 |\mathbf{D}{\dot\bu}^{n+1/2} - \mathbf{C}_1{\dot\bpsi}^{n+1/2}|^2
\\
&& \ge |\mathbf{D}{\bar\bu}^{n+1/2} - \mathbf{C}_1{\bar\bpsi}^{n+1/2}|^2 -
\tfrac13 |{\dot\bu}^{n+1/2}|^2 - \dt^2 |\mathbf{C}_1{\dot\bpsi}^{n+1/2}|^2,
\end{eqnarray*}
so that
\begin{equation}\label{H-mod}
\widetilde H^{n+1/2}  \ge \tfrac1{6} |{\dot\bu}^{n+1/2}|^2 + |\mathbf{D}{\bar\bu}^{n+1/2} - \mathbf{C}_1{\bar\bpsi}^{n+1/2}|^2 +(\alpha-1) \dt^2 |\mathbf{C}_1{\dot\bpsi}^{n+1/2}|^2.
\end{equation}
Note that the term with $\dot\bpsi$ is non-negative for $\alpha\ge1$.

We sum from $n = 0$ to $m$ and note that by the positivity property of $\mathbf{B}(s)$ from Lemma 3.1 and by Lemma 2.3, for $m\dt\le T$,
\begin{eqnarray*}
&&\dt\sum_{n=0}^m
\left\langle 
  \begin{pmatrix}
  \dot \bvarphi^{n}\\ \dot{\bar\bpsi}^{n}
  \end{pmatrix},
 \biggl[  \mathbf{B}(\partial_t^\dt)  
 \begin{pmatrix}
  \dot \bvarphi\\ \dot{\bar\bpsi}
  \end{pmatrix}
  \biggr]^{n}
\right\rangle
\\
&&\qquad
\ge \frac \beta{2eT} 
\dt \sum_{n=0}^{m} \Bigl( |(\partial_t^\dt)^{-1}\dot\bvarphi^{n}|^2 + |(\partial_t^\dt)^{-1}\dot{\bar\bpsi}^{n}|^2 \Bigr) .
\end{eqnarray*}
Here we note that with the BDF2 method, for which $\delta(\zeta)=\tfrac32(1-\zeta)(1-\zeta/3)$,
$$
(\partial_t^\dt)^{-1}\dot\bvarphi^{n} = \tfrac23 \sum_{j=0}^n 3^{-(n-j)} \sum_{k=0}^j \dot \bvarphi^{k} = \tfrac23 \sum_{j=0}^n 3^{-(n-j)} \bvarphi^{j+1/2}.
$$
Hence,
\begin{eqnarray} \label{H-mod-est}
&&\widetilde H^{m+1/2} - \widetilde H^{1/2} +
 \frac \beta{2eT} 
\dt \sum_{n=0}^{m} \Bigl( |\bvarphi^{n+1/2}|^2 + |{\bar\bpsi}^{n+1/2}|^2 \Bigr) 
\\ \nonumber
&&\quad
\le \dt\sum_{n=0}^{m} \Big( (\bar {\dot \bu}^n,\dot{\mathbf{f}}^n- \mathbf{D}^T\bg^n) 
+\langle \dot\bvarphi^{n},\dot\brho^{n}\rangle + 
\langle \dot{\bar\bpsi}^{n}, \mathbf{C}_1^T\bg^{n} +\dot\bsigma^{n} \rangle 
\Big).
\end{eqnarray}

(b) If we first assume that all terms $\dot\brho^n$ and $\mathbf{C}_1^T\bg^{n} +\dot\bsigma^{n}$ are zero, then using the Cauchy-Schwarz inequality and Young's inequality,
and finally the discrete Gronwall inequality, we obtain the estimate at $t=n\dt$,
\[
\widetilde H^{n+1/2} \leq e \left(\widetilde H^{1/2}+ \frac t2 \, \dt\sum_{j=0}^n |{\dot{\mathbf{f}}}^{j}-\mathbf{D}^T\bg^j|^2  \right).
\]

(c) By linearity it remains to study the case where $\bu^0=0$, $\bu^{1}=0$, and all ${\dot{\mathbf{f}}}^{n}-\mathbf{D}^T\bg^n=0$. As in part (c) of the previous proof, we use the energy technique on the transformed equation. The generating power series satisfy the equations
\begin{eqnarray*}
&&\frac{\zeta^{-1}-2+\zeta}{\dt^2} \,\widehat\bu = - \mathbf{D}^T(\mathbf{D}\widehat \bu -
\mathbf{C}_1\widehat\bpsi) - \mathbf{C}_0 \widehat{\dot\bvarphi} 
\\[2mm]
&&\mathbf{B}(s) 
\begin{pmatrix}
\widehat{\dot\bvarphi} \\ \widehat {\dot{\bar\bpsi}}
\end{pmatrix}
= 
\begin{pmatrix}
 \mathbf{C}_0 ^T \widehat {\dot{\bar\bu}} \\
 \mathbf{C}_1 ^T (\mathbf{D}\widehat \bu -
\mathbf{C}_1\widehat\bpsi) - \alpha \dt^2  \mathbf{C}_1 ^T \mathbf{C}_1 \frac{\zeta^{-1}-2+\zeta}{\dt^2}\,\widehat\bpsi)
\end{pmatrix}
+ 
\begin{pmatrix}
\widehat {\dot\brho} \\ \widehat {\dot\bsigma} +  \mathbf{C}_1^T\widehat{\dot\bg}
\end{pmatrix},
\end{eqnarray*}
where $s=\delta(\zeta)/h$ and
$$
\widehat {\dot{\bar\bu}}  = \frac{\zeta^{-1}-\zeta}{2\dt}\, \widehat\bu, \qquad
\widehat {\dot{\bar\bpsi}}  = \frac{\zeta^{-1}-\zeta}{2\dt}\, \widehat\bpsi.
$$
We take the inner product with $\widehat {\dot{\bar\bu}} $ in the interior equation and with $\begin{pmatrix}
\widehat{\dot\bvarphi} \\ \widehat {\dot{\bar\bpsi}}
\end{pmatrix}$
in the boundary equation, sum up and take the real part. This gives
\begin{eqnarray*}
&&\frac2\dt \, \Re\frac{\zeta^{-1}-2+\zeta}{\zeta^{-1}-\zeta} \bigl|\widehat {\dot{\bar\bu}}\bigr|^2  
+ \Re \left\langle 
\begin{pmatrix}
\widehat{\dot\bvarphi} \\ \widehat {\dot{\bar\bpsi}}
\end{pmatrix},\mathbf{B}(s) 
\begin{pmatrix}
\widehat{\dot\bvarphi} \\ \widehat {\dot{\bar\bpsi}}
\end{pmatrix}
\right\rangle
\\
&&+ \ \frac{1}{2\dt} \,\Re (\zeta^{-1}-\zeta) \, |\mathbf{D}\widehat \bu -
\mathbf{C}_1\widehat\bpsi|^2 
+\alpha\dt^2\, \frac2\dt \, \Re\frac{\zeta^{-1}-2+\zeta}{\zeta^{-1}-\zeta} |\mathbf{C}_1\widehat\bpsi|^2 
\\
&&= \langle \widehat{\dot\bvarphi} , \widehat{\dot\brho} \rangle +
 \langle\widehat {\dot{\bar\bpsi}}, \widehat{\dot\bsigma}+  \mathbf{C}_1^T\widehat{\dot\bg} \rangle.
\end{eqnarray*}
For $|\zeta|<1$ we have 
$$
 \Re\frac{\zeta^{-1}-2+\zeta}{\zeta^{-1}-\zeta} >0, \qquad
 \Re (\zeta^{-1}-\zeta) >0,
$$
and hence we conclude, by the same arguments as at the end of part (c) of the proof of Lemma~\ref{lemma:discrete-E}, that
\begin{eqnarray*}
&& \dt\sum_{j=0}^n \left( |\dot\bvarphi^{j}|^2 +  |\dot{\bar\bpsi}^{j}|^2 \right)
\\
&& \leq 2(e/\beta)^2 \max(t^2,t^6) \dt \sum_{j=0}^n \left(|(\partial_t^\dt)^2\dot\brho^{j}|^2 + |(\partial_t^\dt)^2(\dot\bsigma+ \mathbf{C}_1^T{\dot\bg})^j|^2 \right) .
\end{eqnarray*}
We  return to the bound (\ref{H-mod-est}), where we use a Cauchy-Schwarz inequality on the right-hand side and insert the above bound to obtain
$$
\widetilde H^{n+1/2} \leq  C \max(t^2,t^6) \dt\sum_{j=0}^n \left(|(\partial_t^\dt)^2\dot\brho^{j}|^2 + |(\partial_t^\dt)^2(\dot\bsigma+ \mathbf{C}_1^T{\dot\bg})^j|^2 \right) .
$$
(d) By linearity, combining the estimates of (b) and (c) and recalling (\ref{H-mod}) gives the stated result.
\qed
\end{proof}

\subsection{Boundary functions}

\begin{lemma}\label{lemma:discrete-Phi}
The boundary functions are bounded at $t=n\dt$ by
\[
\begin{split}  
&\sum_{j=0}^n (|\bvarphi^{j+1/2}|^2 + |\bar\bpsi^{j+1/2}|^2)  
\\ &\leq C\max(t^2,t^6)  \sum_{j=0}^{n-1} \left(|\dot{\mathbf{f}}^j|^2+|\dot \bg^{j+1/2}|^2+ |\ddot\brho^{j+1/2}|^2 + |\ddot\bsigma^{j+1/2}|^2 \right) ,
\end{split}
\]  
where $C$ is independent of $h$, $\dt$, and $n$.
\end{lemma}
\begin{proof}
  We separate the three cases (i) $\bu^0 = 0$, $\bv^0 = 0$, $\mathbf{f}^{j+1/2} = 0$ and $\bg^{j} = 0$, (ii) $\brho^{j+1/2} = 0$, $\bsigma^{j+1/2} = 0$ and $\bg^j = 0$, and (iii) all inhomogeneities and initial values vanish except for  arbitrary $\bg^j$.
In the case (i) an estimate of the temporal $\ell_2$ norms of $\bvarphi^{j+1/2}$ and $\bar\bpsi^{j+1/2}$ is given in \eqref{eq:E_L2_phi-disc}. In the case (ii) such an estimate follows from \eqref{H-mod-est}. The case (iii) is proved by an extension of part (c) of the proof of Lemma~\ref{lemma:discrete-E}, similar to the proof of 
Lemma~\ref{lemma:semiPhi}.
\qed\end{proof}

\section{Error bound for the full discretization}
We proceed in the same way as for the semidiscretization in 
Section~\ref{sec:err-semi}.
We first rewrite the fully discrete equations in their variational formulation:
find $u_h^n \in U_h$, $v_h^n, v_h^{n+1/2} \in V_h$, $\varphi_h^{n+1/2} \in \Phi_h$, $\psi_h^n \in \Psi_h$ (and $\bar\psi_h^{n+1/2}=\half(\psi_h^{n+1}+\psi_h^n)$ and
$\dot\psi_h^{n+1/2}=\frac1{\dt}(\psi_h^{n+1}-\psi_h^n)$)
such that
\begin{equation}\label{eq:disc-var}
\begin{aligned} 
&  \tfrac2\dt( v_h^{n+1/2}- v_h^n, z_h) = -\half(u_h^n,\nabla\cdot z_h) + \half (\nabla u_h^n,z_h) +\half\langle\psi_h^n,\gamma z_h\cdot n\rangle_\Gamma
\\[1mm]
&  \tfrac1\dt( u_h^{n+1}-u_h^n, w_h) = -\half(v_h^{n+1/2},\nabla w_h) + \half (\nabla\cdot v_h^{n+1/2},w_h) -\half\langle\varphi_h^{n+1/2},\gamma w_h\rangle_\Gamma
\\
&\qquad\qquad\qquad\qquad \qquad\qquad\qquad \qquad +(f(t_{n+1/2}),w_h)
\\[1mm]
&  \tfrac2\dt(v_h^{n+1}-v_h^{n+1/2}, z_h) = -\half(u_h^{n+1},\nabla\cdot z_h) + \half (\nabla u_h^{n+1},z_h) +\half\langle\psi_h^{n+1},\gamma z_h\cdot n\rangle_\Gamma
\\[2mm]
& \left\langle
  \begin{pmatrix}
    \xi_h\\\eta_h
  \end{pmatrix}, \Bigl[ {B} (\partial_t^\dt) 
  \begin{pmatrix}
    \varphi_h \\ \bar\psi_h
  \end{pmatrix}\Bigr]^{n+1/2}
\right\rangle_\Gamma = \half\langle\xi_h,\gamma u_h^{n+1/2}\rangle_\Gamma - \half \langle \gamma v_h^{n+1/2}\cdot n, \eta_h\rangle_\Gamma 
\\
&\qquad\qquad\qquad\qquad \qquad\qquad\qquad \qquad
-\alpha\dt^2 \langle \dot\psi_h^{n+1/2}, \eta_h\rangle_\Gamma 
\end{aligned}
\end{equation}
for all $w_h \in U_h$, $z_h \in V_h$, $\xi_h \in \Phi_h$, and  $\eta_h\in \Psi_h$.

We consider the defects obtained when we insert the projected exact solution $(P_hu,P_hv,P_h\varphi,P_h\psi)$ into the variational formulation of the fully discrete scheme. Instead of $P_hv(t_{n+1/2})$ we insert  $P_h \widetilde v^{n+1/2}$ with $\widetilde v^{n+1/2} = v(t_{n+1/2})- \frac18 \dt^2 \ddot v(t_{n+1/2})$, chosen such that $\widetilde v^{n+1/2} = v(t_n)+\half\dt \,\dot v(t_n)+ O(\dt^3)$ and 
$v(t_{n+1})=  \widetilde v^{n+1/2} + \half\dt \,\dot v(t_{n+1})+O(\dt^3)$. The arising defects in (\ref{eq:disc-var}) then consist of terms that are already present in the defects of the semidiscretization and additional terms that are $O(\dt^2)$ in the case of a temporally smooth solution. For the interior equations this is obtained from a simple Taylor expansion, for the boundary equations it follows from the known error bound (\ref{cq-conv}) of convolution quadrature \cite{Lub94}. We thus have $O(h+\dt^2)$ consistency errors in the appropriate norms. With the discrete stability lemmas from Section~\ref{sec:stab-disc} we then obtain, by the same arguments that we used for the semidiscrete case, the following error bound for the full discretization.

\begin{theorem}
Assume that the initial values and the inhomogeneity of the wave equation  \eqref{eq:wave} have their support in $\Omega$. Let the initial values for the semi-discretization be chosen as $u_h(0) = P_h u(0)$ and $v_h(0) = P_h v (0)$, where $P_h$ denotes the $L_2(\Omega)$-orthogonal projection onto the finite element spaces. If the solution of the wave equation is sufficiently smooth, then the error of the FEM \& BEM  \& leapfrog \& convolution quadrature full discretization \eqref{eq:disc-var}, under the CFL condition $(\ref{cfl})$ and with the stability parameter satisfying $(\ref{alpha})$, is bounded at $t=n\dt$ by
\[
\begin{split}  
&\|u_h^n-u(t)\|_{L_2(\Omega)} + \|v_h^n-v(t)\|_{L_2(\Omega)^3}\\ &+
\left(\dt\sum_{j=0}^{n-1} \|\varphi_h^{j+1/2}-\varphi(t_{j+1/2})\|^2_{H^{-1/2}(\Gamma)}+ \|\bar\psi_h^{j+1/2}-\psi(t_{j+1/2})\|^2_{H^{1/2}(\Gamma)}\right)^{1/2}
\\
&
\leq C(t) (h+\dt^2),
\end{split}
\]
where the constant $C(t)$ grows at most polynomially with $t$.
\qed
\end{theorem}

\def\cprime{$'$}

\end{document}